\documentclass[10pt, a4paper, reqno]{amsart}

\usepackage{amsmath, amstext, amsthm, amsopn, amssymb, amscd, amsxtra, amsfonts, paralist}
\usepackage{graphicx, diagrams}
\usepackage[mathcal]{euscript}

\usepackage{fancyhdr, paralist, epsfig}

\usepackage[mathcal]{euscript}
\usepackage[english]{babel}
\usepackage[latin1]{inputenc}

\newtheoremstyle{normal}{5pt}{5pt}{\normalfont}{}{\bfseries}{}{0.4em}{\bfseries{\thmname{#1}\thmnumber{ #2}.\thmnote{ \hspace{0.5em}(#3)\newline}}}

\newtheoremstyle{kursiv}{5pt}{5pt}{\itshape}{}{\bfseries}{}{0.4em}{\bfseries{\thmname{#1}\thmnumber{ #2}.\thmnote{ \hspace{0.5em}(#3)\newline}}}

\theoremstyle{kursiv}
\newtheorem*{thmA}{Theorem A}
\newtheorem*{thmB}{Theorem B}
\newtheorem{theorem}{Theorem}[section]

\newtheorem{lemma}[theorem]{Lemma}
\newtheorem{proposition}[theorem]{Proposition}
\newtheorem{observation}[theorem]{Observation}

\theoremstyle{normal}
\newtheorem{definition}[theorem]{Definition}
\newtheorem{example}[theorem]{Example}
\newtheorem{remark}[theorem]{Remark}

\newcommand{\ind}[1]{\operatorname{ind}_{#1}\nolimits}
\newcommand{\proj}[1]{\operatorname{proj}_{#1}\nolimits}
\newcommand{\supp}{\operatorname{supp}\nolimits}
\newcommand{\mybinom}[2]{\Bigl(\begin{array}{@{}c@{}}#1\\#2\end{array}\Bigr)}
\renewcommand{\Re}{\operatorname{Re}\nolimits}
\newcommand{\id}{\operatorname{id}\nolimits}
\newcommand{\co}{\operatorname{co}\nolimits}
\newcommand{\ptw}{\operatorname{ptw}\nolimits}
\newcommand{\I}{\operatorname{ind}\nolimits}
\newcommand{\s}{\operatorname{sum}\nolimits}
\newcommand{\lin}{\operatorname{span}\nolimits}
\newcommand{\SR}{\operatorname{\mathcal{S}(\mathbb{R})}\nolimits}
\newcommand{\CPa}{\operatorname{(\text{CP}_1)}\nolimits}
\newcommand{\CPb}{\operatorname{(\text{CP}_2)}\nolimits}
\newcommand{\ACP}{\operatorname{(\text{ACP})}\nolimits}

\newcommand{\T}{(T(t))_{t\geqslant0}}
\newcommand{\Se}{(S(t))_{t\geqslant0}}

\renewcommand{\epsilon}{\varepsilon}

\begin{document}

\title{Asymptotics of Evolution Equations\\beyond Banach spaces}

\author{Birgit Jacob\hspace{1pt}\MakeLowercase{$^{\text{a,1}}$} and Sven-Ake Wegner\hspace{0.5pt}\MakeLowercase{$^{\text{b}}$}}

\renewcommand{\thefootnote}{}
\hspace{-1000pt}\footnote{\hspace{5.5pt}2010 \emph{Mathematics Subject Classification}: Primary 47D06; Secondary 46A04, 34G10.}
\hspace{-1000pt}\footnote{\hspace{5.5pt}\emph{Key words and phrases}: semigroups on locally convex space, asymptotic behavior, exponential stability.\vspace{1.6pt}}

\hspace{-1000pt}\footnote{\hspace{-6.8pt}$^{\text{a},1}$\,Corresponding author: Bergische Universit\"at Wuppertal, FB C -- Mathematik, Gau\ss{}stra\ss{}e 20, 42119 Wuppertal,\linebreak\phantom{x}\hspace{12.5pt}Germany, Phone:\hspace{1.2pt}\hspace{1.2pt}+49\hspace{1.2pt}(0)\hspace{1.2pt}202\hspace{1.2pt}/\hspace{1.2pt}439\hspace{1.2pt}-\hspace{1.2pt}2527, Fax:\hspace{1.2pt}\hspace{1.2pt}+49\hspace{1.2pt}(0)\hspace{1.2pt}202\hspace{1.2pt}/\hspace{1.2pt}439\hspace{1.2pt}-\hspace{1.2pt}3724, E-Mail: bjacob@uni-wuppertal.de.\vspace{1.6pt}}

\hspace{-1000pt}\footnote{$^{\text{b}}$\,Bergische Universit\"at Wuppertal, FB C -- Mathematik, Gau\ss{}stra\ss{}e 20, 42119 Wuppertal, Germany, Phone:\hspace{1.2pt}\hspace{1.2pt}+49\hspace{1.2pt}(0)\linebreak\phantom{x}\hspace{13.2pt}202\hspace{1.2pt}/\hspace{1.2pt}439\hspace{1.2pt}-\hspace{1.2pt}2531, E-Mail: wegner@math.uni-wuppertal.de.}

\begin{abstract}
We study the asymptotics of strongly continuous operator semigroups defined on locally convex spaces in order to develop a stability theory for solutions of evolution equations beyond Banach spaces. In the classical case, there is only little choice for a semigroup's speed in approaching zero uniformly. Indeed, if a strongly continuous semigroup on a Banach space converges to zero uniformly at any speed then it converges already uniformly at exponential speed. Semigroups with this property are said to be exponentially stable. Leaving the Banach space setting, the situation changes entirely; for instance convergence to zero at a speed faster than any polynomial but not exponentially fast is possible. In this article we establish concepts of stability which refine the classical notions and allow to grasp the different kinds of asymptotic behavior. We give characterizations of the new properties, study their relations and consider generic examples like multiplication semigroups and shifts. In addition we apply 
our results to the transport and the heat equation on classical Fr\'{e}chet function spaces.
\end{abstract}

\maketitle


\section{Introduction}\label{Intro}

Many phenomena appearing in nature are concerned with the evolution of measurable or computable variables, which vary in time and for instance describe the state of a given physical system. The physical laws governing the evolution of the systems, typically can be modeled by evolution equations. In this article we study equations of the form
$$\ACP\;\;\;
\begin{cases}
\,\frac{d}{dt}u(t)=Au(t)\text{ for }t\geqslant0,\\
\,\hspace{8pt}u(0)=u_0
\end{cases}
$$
where $u\colon[0,\infty)\rightarrow X$ is a vector valued function and $A\colon D(A)\rightarrow X$, $D(A)\subseteq X$ is a linear operator.
\smallskip
\\There exist various approaches to solve the abstract Cauchy problem above. In the sequel we restrict on the method of operator semigroups which allows to deduce from properties of $A$ the existence of a family $\T$ of operators on $X$ such that the classical solutions of $\ACP$ are given by the trajectories $u=T(\cdot)u_0$ whenever $u_0$ belongs to $D(A)$. If $X$ is a Banach space, the famous Hille-Yosida theorem characterizes those operators $A$ which generate a so-called $C_0$-semigroup; a type of semigroup which appears in many important instances of $\ACP$. We refer to the monograph \cite{EngelNagelOne} of Engel, Nagel for more details. If $X$ is a locally convex space, the class of $C_0$-semigroups allows for generalizations in several different directions. Accordingly, there exists a variety of Hille-Yosida type generation theorems corresponding to different continuity resp.~boundedness properties of the semigroups. To mention a small sample, we refer to Yosida \cite{Yosida}, Miyadera \cite{Miyadera},
 K{\=o}mura \cite{Komura}, Vuvunikyan \cite{Vuvunikyan}, Choe \cite{Choe}, Albanese, K\"uhnemund \cite{AK}, Babalola \cite{Babalola} and K\"uhnemund \cite{Kuehnemund-Hille-Yosida}. More details and references on generation results can be found in the introduction of Doma\'{n}ski, Langenbruch \cite{DoLa}.
\smallskip
\\In this article we study the asymptotic behavior of solutions of $\ACP$ for large time $t>0$ if $X$ is a locally convex space. We assume that the solutions of $\ACP$ are given by a semigroup $\T$ on $X$, which will therefore be the starting point for our asymptotic analysis. Moreover, our focus is on the case where $\lim_{t\rightarrow\infty}T(t)=0$ holds in a sense which has to be precised. For $X$ a Banach space, the question above has been studied extensively. We refer to Engel, Nagel \cite[Chapter V]{EngelNagelOne} and van Neerven \cite{vanNeerven} for a detailed exposition of the Banach space theory. Here, we want to emphasize the following: If one considers the asymptotics of the trajectories of \textit{all} points of $X$, then there is only little choice for the semigroup's speed in approaching zero as the following theorem illustrates, see Engel, Nagel \cite[V.1.2]{EngelNagelOne}, Gor\-ba\-chuk, Gor\-ba\-chuk \cite[Proposition 3]{GorbachukGorbachuk}. 

\begin{thmA} For a $C_0$-semigroup $\T$ on a Banach space $(X,\|\cdot\|)$ the following are equivalent.
\begin{itemize}
\item[(i)] $\lim_{t\rightarrow\infty}\|T(t)\|_{L(X)}=0$.\smallskip

\item[(ii)] $\exists\:\omega>0\colon \lim_{t\rightarrow\infty}e^{\omega{}t}\|T(t)\|_{L(X)}=0$.\smallskip

\item[(iii)] $\exists\:\omega>0\;\forall\:x\in X\colon \lim_{t\rightarrow\infty}e^{\omega{}t}\|T(t)x\|=0$.\smallskip

\item[(iv)] $\forall\:x\in X\;\exists\:\omega>0\colon \lim_{t\rightarrow\infty}e^{\omega{}t}\|T(t)x\|=0$.\smallskip

\item[(v)]$\exists\:w\in C[0,\infty),\,\text{s.th.}\,\lim_{t\rightarrow\infty}w(t)=\infty,\,\forall\:x\in X\colon \lim_{t\rightarrow\infty}w(t)\|T(t)x\|=0$.
\end{itemize}
\end{thmA}

Consequently, a $C_0$-semigroup which converges to zero pointwise arbitrarily slow but with uniform speed converges already uniformly to zero with exponential speed provided that $X$ is Banach space. The so-called exponentially stable semigroups can further be characterized by integrability conditions on their trajectories, cf.~Datko \cite[p.~615]{Datko}, Pazy \cite[p.~119]{Pazy}, Zabczyk \cite[Theorem 5.1]{Zabczyk}, Littman \cite[Theorem 2]{Littman}, Gorbachuk, Gorbachuk \cite[Theorem 1]{GorbachukGorbachuk}.

\begin{thmB} For a $C_0$-semigroup $\T$ on a Banach space $(X,\|\cdot\|)$ the following are equivalent.
\begin{itemize}
\item[(i)] $\T$ satisfies the conditions of Theorem A.\smallskip

\item[(ii)] $\forall\text{ or, equivalently, }\exists\:\beta\geqslant1\;\forall\:x\in X\colon\int_0^\infty \|T(t)x\|^{\beta}dt<\infty$.\smallskip

\item[(iii)] $\forall\:x\in X\;\exists\:\beta\geqslant1\colon \int_{0}^{\infty}\|T(t)x\|^{\beta}dt<\infty$.\smallskip

\item[(iv)] $\exists\:w\in C[0,\infty)\,\text{strictly increasing s.th.}\,w(0)=0,\,\forall\:x\in X\colon \int_0^{\infty}w(\|T(t)x\|)dt<\infty$.
\end{itemize}
\end{thmB}

If $X$ is a locally convex space, then the above equivalences fail in general. In particular, it may happen that a $C_0$-semigroup approaches zero uniformly but without exponential speed. In Section \ref{NotationAndMainResults} we therefore introduce several concepts of stability and establish their hierarchy under mild assumptions on $X$. For barrelled, Baire and (Mackey) complete spaces, respectively, we prove characterizations of the properties using conditions similar to those explained in Theorem's A and B. In Section \ref{Examples} we provide examples which illustrate that the classes of stable semigroups considered in this article are distinct in general and that our characterizations in Section \ref{NotationAndMainResults} are sharp in the sense that the assumptions on the underlying space cannot be dropped. In Section \ref{Proofs} we give the proofs of our results. Finally, in Section \ref{Applications}, we return to the abstract Cauchy problem $\ACP$ above and study the transport and the heat 
equation on the Schwartz space of rapidly decreasing functions resp.~on a Fr\'{e}chet space introduced by Miyadera \cite{Miyadera}.
\smallskip
\\For the theory of semigroups on Banach spaces we refer to Engel, Nagel \cite{EngelNagelOne}, for the locally convex case to Yosida \cite{Yosida} and Choe \cite{Choe}. For the general theory of locally convex spaces we refer to Meise, Vogt \cite{MeiseVogtEnglisch}, Jarchow \cite{Jarchow} and Floret, Wloka \cite{FloretWloka}.

\section{Notation and Main Results}\label{NotationAndMainResults}

Let $X$ be a locally convex space. Unless specified otherwise we tacitly assume that $X$ is Hausdorff. By $\Gamma_X$ we denote a system of continuous seminorms determining the topology of $X$ and by $L(X)$ the space of all linear and continuous operators from $X$ into itself. The strong operator topology $\tau_s$ on $L(X)$ is determined by the family of seminorms
$$
q_x(S):=q(Sx), \; S\in L(X),
$$
for each $x\in X$ and $q\in\Gamma_X$; we write $L_s(X)$ in this case. We denote by $\mathcal{B}_X$ the collection of all bounded subsets of $X$. The topology $\tau_b$ of bounded convergence is determined by the seminorms
$$
q_B(S):=\sup_{x\in B}q(Sx), \; S\in L(X),
$$
for $B\in\mathcal{B}_X$ and $q\in\Gamma_X$; in this case we write $L_b(X)$. If $X$ is a Banach space, then $\tau_b$ is the operator norm topology on $L(X)$. The space $X$ is barrelled, if every absolutely convex, absorbing and closed set (v.i.z., every barrel) is a neighborhood of zero in $X$.

\begin{definition}\label{DFN-SEMIGROUP} Let $X$ be a locally convex space. We say that a one parameter family of operators $\T\subseteq{}L(X)$ is a \textit{strongly continuous semigroup} or for short a \textit{$C_0$-semigroup} on $X$, if
\begin{itemize}
\item[(i)] $T(0)=\id_X$,\smallskip
\item[(ii)] $\forall\:s,\,t\geqslant0\colon T(s+t)=T(s)T(t)$,\smallskip
\item[(iii)] $\forall\:t_0\geqslant0,\,x\in X\colon{\lim_{t\rightarrow t_0}}T(t)x=T(t_0)x$\smallskip
\end{itemize}
holds. We say that $\T$ is \textit{exponentially bounded}, if 
\begin{itemize}
\item[(iv)] $\forall\:q\in\Gamma_X\;\exists\:p\in\Gamma_X,\,M\geqslant1,\,\omega\in\mathbb{R}\;\forall\:t\geqslant0,\,x\in X\colon q(T(t)x)\leqslant{}Me^{\omega{}t}p(x)$
\end{itemize}
is valid. Finally we say that $\T$ is \textit{bounded}, if
\begin{itemize}
\item[(v)] $\forall\:B\in\mathcal{B}_X,\,q\in\Gamma_X\:\exists\:C\geqslant0\;\forall\:t\geqslant0\colon q_B(T(t))\leqslant{}C$.
\end{itemize}
\end{definition}

If $X$ is a Banach space, then every $C_0$-semigroup is exponentially bounded. Boundedness in the above sense means exactly that $\T\subseteq L_b(X)$ is bounded. If $X$ is barrelled then the latter is equivalent to the condition in (iv) with $\omega=0$, v.i.z., to equicontinuity of $\T$.

\begin{definition}\label{DFN-STABILITY} Let $X$ be a locally convex space and $\T$ a $C_0$-semigroup on $X$. We say that $\T$ is
\begin{itemize}
\item[(i)] \textit{uniformly exponentially stable}, if
$$
\exists\:\omega>0\;\forall\:B\in\mathcal{B}_X,\,q\in\Gamma_X\colon\lim_{t\rightarrow\infty}q_B(e^{\omega{}t}T(t))=0,
$$
\item[(ii)] \textit{pseudo uniformly exponentially stable}, if
$$
\forall\:q\in\Gamma_X\;\exists\:\omega>0\;\forall\:B\in\mathcal{B}_X\colon\lim_{t\rightarrow\infty}q_B(e^{\omega t}T(t))=0,
$$
\item[(iii)] \textit{strongly exponentially stable}, if
$$
\forall\:x\in X\;\exists\:\omega>0\colon\lim_{t\rightarrow\infty}e^{\omega t}T(t)x=0,
$$
\item[(iv)] \textit{pseudo strongly exponentially stable}, if
$$
\forall\:q\in\Gamma_X,\,x\in X\;\exists\:\omega>0\colon\lim_{t\rightarrow\infty}q(e^{\omega t}T(t)x)=0,
$$
\item[(v)] \textit{super polynomially stable}, if
$$
\forall\:\alpha>1,\,B\in\mathcal{B}_X,\,q\in\Gamma_X\colon\lim_{t\rightarrow\infty}q_B(t^{\alpha}T(t))=0,
$$
\item[(vi)] \textit{uniformly stable}, if 
$$
\forall\:B\in\mathcal{B}_X,\,q\in\Gamma_X\colon\lim_{t\rightarrow\infty}q_B(T(t))=0,
$$
\item[(vii)] \textit{strongly stable}, if 
$$
\forall\:x\in X\colon\lim_{t\rightarrow\infty}T(t)x=0.
$$
\end{itemize}
\end{definition}

For $X$ a Banach space the properties (i)--(vi) above collapse to the classical notion of exponential stability, whereas strong stability is a strictly weaker condition.
\smallskip
\\A sequence $(x_n)_{n\in\mathbb{N}}$ in $X$ is Mackey Cauchy, if there exist $B\in\mathcal{B}_X$ and scalars $b_{n,k}$ tending to zero for $n$ and $k$ both tending to infinity such that $x_n-x_k\in b_{n,k}B$ holds for all $n$ and $k\in\mathbb{N}$. The space $X$ is Mackey complete if every Mackey Cauchy sequence in $X$ is convergent. Note that every sequentially complete space is Mackey complete.

\begin{theorem}\label{MAIN-0}For a Mackey complete, barrelled locally convex space $X$ and an exponentially bounded $C_0$-semigroup we have the following hierarchy
\begin{center}
\includegraphics[width=12cm]{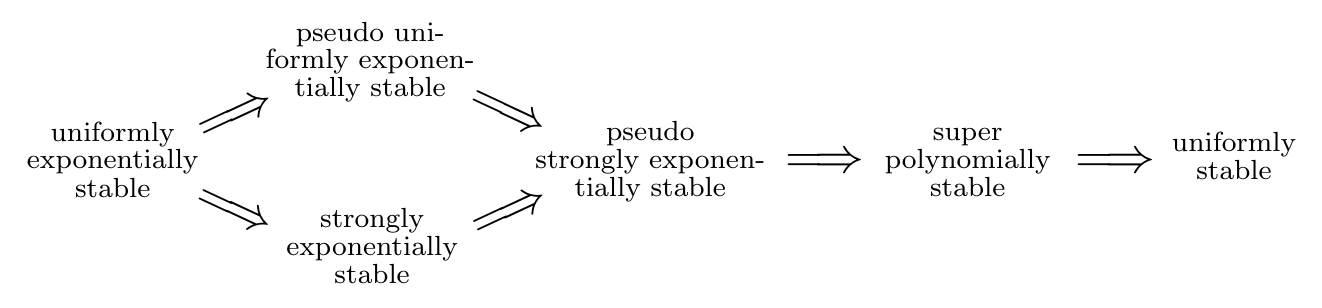}
\end{center}
of stability properties. None of the implications is an equivalence and between the two perpendicular conditions no implication is true in general.
\end{theorem}

The space $X$ is a Baire space if $X$ cannot be written as the countable union of nowhere dense sets.

\begin{theorem}\label{MAIN-1} For a $C_0$-semigroup $\T$ consider the following statements.
\begin{itemize}
\item[(i)] $\T$ is uniformly exponentially stable.\smallskip
\item[(ii)] $\exists\:\omega>0\;\forall\:q\in\Gamma_X\;\exists\:p\in\Gamma_X\colon\lim_{t\rightarrow\infty}\sup_{p(x)\leqslant1}q(e^{\omega{}t}T(t)x)=0$.\smallskip
\item[(iii)] $\exists\:\omega>0\;\forall\:x\in X\colon \lim_{t\rightarrow\infty}e^{\omega t}T(t)x=0$.\smallskip
\item[(iv)] $\T$ is strongly exponentially stable.\smallskip
\item[(v)] $\forall\:B\in\mathcal{B}_X\;\exists\:\omega>0\;\forall\:q\in\Gamma_X\colon \lim_{t\rightarrow\infty}q_B(e^{\omega t}T(t))=0$.\smallskip
\end{itemize}
If $X$ is barrelled, then (i)--(iii) are equivalent. If $X$ is Mackey complete, then (iv) and (v) are equivalent. If $X$ is Baire, then all statements are equivalent.
\end{theorem}

\begin{theorem}\label{MAIN-2} For a $C_0$-semigroup $\T$ consider the following statements.
\begin{itemize}
\item[(i)] $\T$ is super polynomially stable.\smallskip
\item[(ii)] $\exists\:\alpha>1\;\forall\:B\in\mathcal{B}_X,\,q\in\Gamma_X\colon \lim_{t\rightarrow\infty}q_B(t^{\alpha}T(t))=0$.\smallskip
\item[(iii)] $\forall\text{ or, equivalently, }\exists\:\alpha>1\;\forall\:q\in\Gamma_X\;\exists\:p\in\Gamma_X\colon\lim_{t\rightarrow\infty}\sup_{p(x)\leqslant1}q(t^{\alpha}T(t)x)=0$.\smallskip
\item[(iv)]  $\forall\text{ or, equivalently, }\exists\:\alpha>1\;\forall\:x\in X\colon \lim_{t\rightarrow\infty}t^{\alpha}T(t)x=0$.\smallskip
\item[(v)] $\forall\text{ or, equivalently, }\exists\:\beta\geqslant1\;\forall\:q\in\Gamma_X,\,x\in X\colon\int_0^\infty q(T(t)x)^{\beta}dt<\infty$.\smallskip
\item[(vi)] $\forall\:x\in X\;\exists\:\alpha>1\;\colon \lim_{t\rightarrow\infty}t^{\alpha}T(t)x=0$.\smallskip
\item[(vii)] $\forall\:B\in\mathcal{B}_X\;\exists\:\alpha>1\;\forall\:q\in\Gamma_X\colon \lim_{t\rightarrow\infty}q_B(t^{\alpha}T(t))=0$.\smallskip
\end{itemize}
Let $\T$ be exponentially bounded. If $X$ is barrelled, then (i)--(v) are equivalent. If $X$ is Mackey complete, then (vi) and (vii) are equivalent. If $X$ is Baire, then all statements are equivalent.
\end{theorem}

\begin{theorem}\label{MAIN-3} Let $X$ be Mackey complete and $\T$ be a $C_0$-semigroup.
\begin{itemize}
\item[(1)] The following are equivalent.\vspace{2pt}
   \begin{itemize}
   \item[(i)] $\T$ is pseudo uniformly exponentially stable.\smallskip
   \item[(ii)] $\forall\:q\in\Gamma_X\;\exists\:\omega>0\;\forall\:x\in X\colon \lim_{t\rightarrow\infty}q(e^{\omega{}t}T(t)x)=0$.
   \end{itemize}\vspace{4pt}
\item[(2)] The following are equivalent.\vspace{2pt}
   \begin{itemize}
   \item[(iii)] $\T$ is pseudo strongly exponentially stable.\smallskip
   \item[(iv)] $\forall\:q\in\Gamma_X,\,B\in\mathcal{B}_X\;\exists\:\omega>0\colon\lim_{t\rightarrow\infty}q_B(e^{\omega{}t}T(t))=0$.
   \end{itemize}
\end{itemize}
\end{theorem}

The space $X$ is a Montel space if $X$ is barrelled and every bounded subset of $X$ is relatively compact; normed Montel spaces are thus necessarily of finite dimension.

\begin{theorem}\label{MAIN-4} Let $X$ be a Montel space and $\T$ be a bounded $C_0$-semigroup. The following are equivalent.
\begin{itemize}
\item[(i)] $\T$ is uniformly stable.\smallskip
\item[(ii)] $\T$ is strongly stable.
\end{itemize}
\end{theorem}

Uniform exponential stability means that $\lim_{t\rightarrow\infty}e^{\omega{}t}T(t)=0$ holds in $L_b(X)$ for some $\omega>0$. Analogously, uniform stability and super polynomial stability can be rephrased in terms of a convergence statement in $L_b(X)$. In contrast, strong exponential stability and the two \textquotedblleft{}pseudo properties\textquotedblright{} a priori do not allow for an interpretation in a similar manner, since in these definitions $\omega$ may depend on the points of the space or even on the seminorms inducing the topology.

\section{Examples}\label{Examples}

\subsection{Shift Semigroups}

In all examples of this section, $\T$ denotes the right shift operator, i.e.,
$$
(T(t)f)(s)=f(s-t)
$$
for $s-t\in\mathrm{\Omega}$ and $(T(t)f)(s)=0$ otherwise; $f\colon\mathrm{\Omega}\rightarrow\mathbb{C}$ is a function on $\mathrm{\Omega}=\mathbb{R}$, $[0,\infty)$ or $(-\infty,0]$.

\begin{example}\label{SHIFT-1} We consider the right shift $\T$ on the space $C(\mathbb{R})$ endowed with the topology $\tau_{\ptw}$ of pointwise convergence generated by the seminorms $q_s(f)=|f(s)|$, $s\in\mathbb{R}$. It is clear that $\T$ defines a $C_0$-semigroup on the latter space. This semigroup is not even strongly stable in the sense of Definition \ref{DFN-STABILITY}(vii). The same holds on $C(-\infty,0]$, let us thus consider $X=\{f\in C[0,\infty)\:;\:f(0)=0\}$. For a seminorm $q_s$, $s\geqslant0$, a bounded set $B\in\mathcal{B}_X$ and arbitrary $\omega>0$ we compute $\sup_{f\in B}q_s(e^{\omega{}t}(T(t)f)(s))=0$ whenever $t\geqslant{}s$, i.e., the right shift is uniformly exponentially stable on $X$. Let us mention that $\T$ fails to be bounded: Consider $B=\{f_n\:;\:n\geqslant2\}$ where the hat functions
$$
f_n(s)=
\begin{cases}
\;\hspace{32pt}0 &0\leqslant s<1-2/n \text{ or } 1\leqslant s\\
\;n^2(s-(1-2/n)) &1-2/n\leqslant s<1-1/n\\
\;\hspace{12pt}-n^2(s-1) &1-1/n\leqslant s<1
\end{cases}
$$
are inspired by K\"uhnemund \cite[1.6(b)]{Kuehnemund}. For $s\geqslant0$ and $n\geqslant0$ we have $q_s(f_n)=|f_n(s)|\rightarrow0$ for $n\rightarrow\infty$. Thus, $B$ is bounded in $X$ and for $s=1$ we have $\sup_{t\geqslant0}\sup_{f\in B}q_s(T(t)f)\geqslant\sup_{n\geqslant2}|(T(1/n)f_n)(1)|=\sup_{n\geqslant2}f_n(1-1/n)=\infty$, whence $\T$ is unbounded.
\end{example}

The somewhat unintuitive effect of a stable but unbounded $C_0$-semigroup as illustrated above can be attributed to the underlying space' failure of being barrelled, cf.~Remark \ref{REM-1}.

\begin{observation}\label{OBS-1} Let $\T$ be a $C_0$-semigroup on a barrelled space $X$. Assume $\lim_{t\rightarrow\infty}T(t)=0$ in $L_b(X)$. Then $\T$ is bounded.
\end{observation}
\begin{proof} Let $B\in\mathcal{B}$ and $q\in\Gamma_X$ be given. From our assumption on $\T$ it follows that there is $t_0\geqslant0$ such that $q_B(T(t))\leqslant1$ holds for all $t\geqslant{}t_0$. Since $X$ is barrelled, K{\=o}mura \cite[1.1]{Komura} implies that $\T$ is locally equicontinuous, i.e., $\{T(t)\:;\:0\leqslant{}t\leqslant{}t_0\}\subseteq{}L(X)$ is equicontinuous for all $t_0\geqslant0$. In particular, the latter sets are all bounded in $L_b(X)$, v.i.z., $\T$ is locally bounded, and thus there is $C\geqslant0$ such that $q_B(T(t))\leqslant{}C$ for $0\leqslant{}t\leqslant{}t_0$. Consequently, $q_B(T(t))\leqslant1+C$ for all $t\geqslant0$ which shows that $\T$ is bounded.
\end{proof}

In order to avoid pathologies of the type above let us endow our space with a topology which turns it into a barrelled space.

\begin{example}\label{SHIFT-2} Consider the right shift on the space $C(\mathbb{R})$ endowed with the compact open topology $\tau_{\co}$ generated by the seminorms $q_n(f)=\sup_{s\in[-n,n]}|f(s)|$, $n\geqslant1$. Again, $\T$ defines a $C_0$-semigroup on this space which is not even strongly stable. On this space, $\T$ provides an example for a non exponentially bounded semigroup: Put $n=1$ and let $m$, $M$ and $\omega>0$ be given. Select $f\in C(\mathbb{R})$ such that $\supp(f)\subseteq[-(m+2),-m]$, $0\leqslant{}f\leqslant1$, $f(-(m+1))=1$. Compute $p_n(T(m+1)f)=\sup_{s\in[-1,1]}|f(s-(m+1))|\geqslant|f(-(m+1))|=1$ and $Me^{\omega{}t}p_m(f)=0$ as $f|_{[-m,m]}\equiv{}0$. On the space $C(-\infty,0]$ the semigroup $\T$ has the same properties and we consider $X=\{f\in C[0,\infty)\:;\:f(0)=0\}$. As in the case of pointwise convergence $\T$ is uniformly exponentially stable when considered on this space and in particular exponentially bounded by Observation \ref{OBS-1}. Note that $C(\mathbb{R})$ is barrelled by a 
result of Nachbin \cite{Nachbin} and Shirota \cite{Shirota}, see Jarchow \cite[11.7.5]{Jarchow}, and this implies that $X$ is barrelled by a result of Dieudonn\'{e} \cite{Dieudonne}, see Bonet, Perez Carreras \cite[4.3.1]{BPC}.
\end{example}

Next, we consider the shift on the space of compactly supported continuous functions both with pointwise and compact convergence.

\begin{example}\label{SHIFT-3} The right shift $\T$ on $(C_c(\mathbb{R}),\tau_{\co})$ is strongly exponentially stable: Given $f\in C_c(\mathbb{R})$ and $n\in\mathbb{N}$ we have $(T(t)f)|_{[-n,n]}\equiv0$ if $t$ is big enough. On the other hand, $\T$ is not uniformly stable: Put $n=1$ and $B=\{f_j\:;\:j\geqslant1\}$ where $f_j\in C_c(\mathbb{R})$ with $0\leqslant{}f_j\leqslant{}1$ and $f_j(-j)=1$ for all $j\geqslant1$. Since $q_n(f_j)\leqslant1$ holds for all $j$ and $n$ the set $B$ is bounded, but $\sup_{f\in B}q_n(T(t)f)=\sup_{j\geqslant1}\sup_{s\in[-1,1]}|f_j(s-t)|\geqslant|f_t(-t)|=1$ holds for any integer $t$. On $(C_c(-\infty,0],\tau_{\co})$, $(C_c(\mathbb{R}),\tau_{\ptw})$ and $(C_c(-\infty,0]),\tau_{\ptw})$ the same result is true, whereas on $\{f\in C_c[0,\infty)\:;\:f(0)=0\}$ the shift is uniformly exponentially stable w.r.t.~$\tau_{\co}$ and $\tau_{\ptw}$. The first setting considered above shows that both statements of Theorem \ref{MAIN-3} are wrong if Mackey completeness is dropped: Theorem
\ref{MAIN-3}(i) and Theorem \ref{MAIN-3}(iv) imply uniform stability and thus cannot hold whereas Theorem \ref{MAIN-3}(ii) and Theorem \ref{MAIN-3}(iii) both follow from strong exponential stability.
\end{example}

The topologies considered above are not the natural choices for a topology on spaces of compactly supported continuous functions, for instance the spaces above are all incomplete and non-barrelled. Let us thus as a last example consider the shift again on the space of compactly supported functions but with another topology.

\begin{example}\label{SHIFT-4} We endow $C_c(\mathbb{R})$ with the locally convex inductive topology  $\tau_{\I}$ of the inclusion maps $C_K(\mathbb{R})\rightarrow C_c(\mathbb{R})$ where for $K\subseteq\mathbb{R}$ compact $C_K(\mathbb{R})=\{f\in C_c(\mathbb{R})\:;\:\supp f\subseteq K\}$ is endowed with the supremum norm. By Bierstedt \cite[Example 1.7]{Bierstedt1988} the space $(C_c(\mathbb{R}),\tau_{\I})=\ind{n\in\mathbb{N}}C_{[-n,n]}(\mathbb{R})$ is a complete, hence regular, LB-space. A fundamental system of seminorms is given by $q_v(f)=\sup_{s\in\mathbb{R}}v(x)|f(x)|$, $v\in C^+(\mathbb{R})=\{w\in C(\mathbb{R})\:;\:w\geqslant0\}$, see \cite[Proposition 1.13]{Bierstedt1988}. Using the universal property of the inductive limit it follows that $\T$ is a $C_0$-semigroup. For $v\equiv{}1\in C^+(\mathbb{R})$ we have $p_v(T(t)f)=\sup_{s\in\mathbb{R}}|f(s)|$ which shows that $\T$ is not strongly stable. On the space $X=\ind{n}X_n$ with $X_n=\{f\in C_{[0,n]}(\mathbb{R})\:;\:f(0)=0\}$ the $C_0$-semigroup $\T$ is 
also 
not strongly stable: $X$ is continuously included in $(C_c(\mathbb{R}),\tau_{\I})$ whence its topology is finer than those generated by the seminorms $q_v$, $v\in C^+[0,\infty)$. Consider finally the shift $\T$ on the space $(C_c(-\infty,0],\tau_{\I})=\ind{n}C_{[-n,0]}(-\infty,0]$. Again, $\T$ is a $C_0$-semigroup, but now $\T$ is uniformly exponentially stable since by regularity for any bounded set $B$ there exists $n\geqslant1$ such that $B\subseteq C_{[-n,0]}(-\infty,0]$ is norm bounded, and thus $B$ is annihilated by the right shift for all $t\geqslant0$ suitable large.
\end{example}

\subsection{Multiplication Semigroups}

In all examples of this section, $\T$ denotes the multiplication semigroup
$$
T(t)x=(e^{q_jt}x_j)_{j\in\mathbb{N}}
$$
for $t\geqslant0$ and $x=(x_j)_{j\in\mathbb{N}}\subseteq\mathbb{C}$, where we assume that $(q_j)_{j\in\mathbb{N}}\subseteq\mathbb{C}$ satisfies $\sup_{j\in\mathbb{N}}\Re q_j<\infty$. On all spaces which we consider in the subsequent examples, the multiplication semigroup $\T$ is uniformly exponentially stable, if all $\Re q_j$ are negative and bounded away from zero. If on the other hand $\Re q_j\geqslant0$ holds for some $j$, then $\T$ will not even be strongly stable. We thus concentrate below on the remaining cases, i.e., we assume that $\Re q_j<0$ for all $j$ and $\limsup_{j\rightarrow\infty}\Re q_j=0$. The second assumption can clearly be dropped in all positive results.
\smallskip
\\We start with considering $\T$ on the standard example of a non-barrelled normed space.

\begin{example}\label{MULT-1} Denote by $\varphi$ the space of finite sequences which we endow in this example with the supremum norm and consider the multiplication semigroup $\T$ under the assumptions we made above. Then $\T$ is strongly exponentially stable: Given $x$ select $0<\omega<c:=|\max_{j\in\supp{}x}\Re q_j|$ and estimate $\|e^{\omega{}t}T(t)x\|_{\infty}\leqslant{}\|x\|_{\infty}e^{(\omega-c)t}$. On the other hand, for any $\omega>0$ we can choose $j_0$ with $\omega+\Re q_{j_0}>0$, that is, $\omega$ cannot be selected independently of $x$. Even worser, with $B=\{(\delta_{j,j_0})_{j\in\mathbb{N}}\:;\:j_0\in\mathbb{N}\}$ we get that $\sup_{x\in B}\|T(t)x\|\geqslant\sup_{j_0\in\mathbb{N}}e^{t\Re{}q_{j_0}}\geqslant e^0=1$ and $\T$ cannot be uniformly stable.
\end{example}

Firstly, the previous example shows that a semigroup can be strongly exponentially stable without being uniformly exponentially stable nor satisfying  Theorem \ref{MAIN-1}(v). Secondly, in the previous example, Theorem \ref{MAIN-2}(vi) holds but $\T$ is not super polynomially stable. In addition, it can be seen directly that Theorem \ref{MAIN-2}(vii) is not satisfied. Consequently we get that the equivalences in the second and third statements of Theorem \ref{MAIN-1} and Theorem \ref{MAIN-2} fail in general.
\smallskip
\\In the next example we endow the space $\varphi$ from above with a more natural topology.

\begin{example}\label{MULT-2} We endow $\varphi=\oplus_{j\in\mathbb{N}}\mathbb{C}$ with the direct sum topology $\tau_{\s}$, which yields a complete barrelled space, cf.~Meise, Vogt \cite[24.4]{MeiseVogtEnglisch}. As above we consider the multiplication semigroup $\T$. Let $B$ be a bounded set in $(\varphi,\tau_{\s})$. Then by \cite[24.2]{MeiseVogtEnglisch} there exists $N\geqslant1$ such that $\supp{}b\subseteq\{1,\dots,N\}$ holds for all  $b\in B$. Thus, with $0<\omega<|\max_{j=1,\dots,N}\Re q_j|$ it follows that $q_B(e^{\omega{}t}T(t))$ converges to zero for any continuous seminorm $q$ on $(\varphi,\tau_{\s})$ if $t\rightarrow\infty$ and therefore $\T$ is strongly exponentially stable and then also pseudo strongly exponentially stable. On the other hand, consider the continuous seminorm $q(x)=\sum_{j\in\mathbb{N}}|x_j|$ on $(\varphi,\tau_{\s})$, cf.~\cite[Definition on p.~276]{MeiseVogtEnglisch}. For given $\omega>0$ select $j_0$ such that $\omega+\Re q_{j_0}>0$ and put $x=(\delta_{j,j_0})_{j\in\mathbb{N}}
$. 
Then, $q(e^{\omega{}t}T(t)x)=e^{(\omega+\Re q_{j_0})t}$ does not converge to zero, i.e., $\T$ is not pseudo uniformly exponentially stable and thus also cannot be uniformly exponentially stable. Therefore, this example on the one hand shows that for the equivalence of all properties of Theorem \ref{MAIN-1} it is not enough if the underlying space is barrelled and Mackey complete; in fact the above space is even a complete LB-space. Let us add that the $C_0$-semigroup $\T$ with $(T(t)f)(s)=e^{q(s)t}f(s)$ for $t\geqslant0$ and $s\in\mathbb{R}$ defined on the space $(C_c(\mathbb{R}),\tau_{\I})$ is a continuous variant of Example \ref{MULT-2} having exactly the same properties.
\end{example}

In the remainder of this section we study semigroups on K\"othe echelon spaces. For this purpose let $A=(a_{j,k})_{j,k\in\mathbb{N}}$ be a K\"othe matrix, i.e., $0\leqslant{}a_{j,k}\leqslant a_{j,k+1}$ holds for all $j$, $k\in\mathbb{N}$ and for every $j\in\mathbb{N}$ there exists $k\in\mathbb{N}$ such that $a_{j,k}>0$ holds. We consider the complex K\"othe echelon spaces, see \cite[Section 27]{MeiseVogtEnglisch},
\begin{eqnarray*}
\lambda^{\infty}(A)&=&\big\{x\in\mathbb{C}^\mathbb{N}\:;\:\forall\:k\in\mathbb{N}\colon\:\|x\|_k=\sup_{j\in\mathbb{N}} a_{j,k}|x_j|<\infty\big\},\\
c_0(A)&=&\big\{x\in\lambda^{\infty}(A)\:;\:\forall\:k\in\mathbb{N}\colon\:\lim_{j\rightarrow\infty}a_{j,k}|x_j|=0\big\}
\end{eqnarray*}
which are Fr\'echet spaces with respect to the fundamental system $(\|\cdot\|_k)_{k\in\mathbb{N}}$ of seminorms. We consider the multiplication semigroup on $c_0(A)$, where it is  equicontinuous since $|e^{q_jt}|=e^{t\Re q_j}\leqslant1$ holds for any $t>0$ and any $j\in\mathbb{N}$; in particular $\T$ is exponentially bounded. That $\T$ is a $C_0$-semigroup can be seen as follows: Let $x\in c_0(A)$ be fixed, $k\in\mathbb{N}$ be arbitrary. We show that $\|T(t)x-x\|_k\rightarrow0$ holds for $t\rightarrow 0$. Let $\epsilon>0$ be given. Put $\epsilon_0=\epsilon/(C+1+\|x\|_k)$, where $C=\sup_{0\leqslant t\leqslant1}\sup_{j\in\mathbb{N}}|e^{tq_j}-1|<\infty$. Since $x$ belongs to $c_0(A)$ there exists $J\in\mathbb{N}$ such that $a_{j,k}|x_j|<\epsilon_0$ holds for all $j\geqslant J$. We fix $1\leqslant j\leqslant J$. Since $e^{tq_j}\rightarrow1$ holds for $t\rightarrow0$ there exists $0<t_j\leqslant1$ with $|e^{tq_j}-1|<\epsilon_0$ for all $0<t\leqslant t_j$. We put $t_0:=\min_{j=1,\dots,J}t_j\in\,(0,1]$ and estimate
\begin{eqnarray*}
\|T(t)x-x\|_k & = & \sup_{j\in\mathbb{N}}a_{j,k}\big|e^{tq_j}x_j-x_j|\\
              & \leqslant & \sup_{j=1,\dots,J}a_{j,k}\big|e^{tq_j}-1\big|\,|x_j|+\sup_{j>J}a_{j,k}\big|e^{tq_j}-1\big|\,|x_j|\\
              & \leqslant & \epsilon_0\sup_{j=1,\dots,J}a_{j,k}|x_j|+C\sup_{j>J}a_{j,k}|x_j|\\
              & \leqslant & \epsilon_0\,\|x\|_k+C\,\epsilon_0<\epsilon_0(\|x\|_k+C+1)=\epsilon
\end{eqnarray*}
for arbitrary $0<t\leqslant t_0$. Thus, $T(\cdot)x$ is continuous at zero for arbitrary $x$. By \cite[Remark 1(iii)]{ABR}, $\T$ is a $C_0$-semigroup.

\begin{example}\label{MULT-3} The multiplication semigroup $\T$ on $c_0(A)$ is strongly stable. In order to see this, first fix $x\in\varphi$ and $k\in\mathbb{N}$. Let $\epsilon>0$ be given and put $C=\max_{j\in\supp x}\Re q_j<0$. We select $t_0\geqslant0$ such that $\|x\|_k e^{t_0C}<\epsilon$. For $t\geqslant t_0$ we estimate
\begin{eqnarray*}
\|T(t)x\|_k&=&\sup_{j\in\mathbb{N}}a_{j,k}|e^{tq_j}x_j|\;=\sup_{j\in\supp x}a_{j,k}|e^{tq_j}|\,|x_j|\\
&\leqslant& \sup_{j\in\supp x}a_{j,k}|x_j|\cdot\sup_{j\in\supp x}|e^{tq_j}| \;=\; \|x\|_k \max_{j\in\supp x}|e^{tq_j}|\\
           &=& \|x\|_k\,e^{t\max_{j\in\supp x} \Re q_j}=\|x\|_k\,e^{tC}<\epsilon.
\end{eqnarray*}
Let now $x\in c_0(A)$, $k\in\mathbb{N}$ and $\epsilon>0$ be arbitrary. We select $y\in\varphi$ such that $\|x-y\|_k<\frac{\epsilon}{2}$. By the above there exists $t_0\geqslant0$ such that $\|T(t)y\|_k<\frac{\epsilon}{2}$ holds for all $t\geqslant t_0$. We estimate
$$
\|T(t)x\|_k\leqslant\|T(t)(x-y)\|_k+\|T(t)y\|_k<\frac{\epsilon}{2}+\frac{\epsilon}{2}=\epsilon
$$
since $\sup_{j\in\mathbb{N}}\Re q_j\leqslant0$ implies
$$
\|T(t)(x-y)\|_k\leqslant \|x-y\|_k\sup_{j\in\mathbb{N}}|e^{tq_j}|=\|x-y\|_k\,e^{t\sup_{j\in\mathbb{N}}\Re q_j}\leqslant\|x-y\|_k
$$
for arbitrary $t\geqslant t_0$. On the other hand, $\T$ does not satisfy the equivalent properties of Theorem \ref{MAIN-1}: For any $\omega>0$ there is $j_0$ such that $\omega+\Re q_{j_0}>0$ holds. With $x=(\delta_{j,j_0})_{j\in\mathbb{N}}$ and $k$ such that $a_{j_0,k}>0$ it follows that $\|e^{\omega t}T(t)x\|_k=a_{j_0,k}e^{(\omega+\Re q_{j_0})t}$ does not converge to zero for $t\rightarrow\infty$. 
\end{example}

In Example \ref{MULT-3} we considered an arbitrary K\"othe matrix $A$. In the sequel, we deduce properties of $\T$ from properties we require for $A$. We start with the space $\mathbb{C}^{\mathbb{N}}$ of all sequences.

\begin{example}\label{MULT-4} For $A=(a_{j,k})_{j,k\in\mathbb{N}}$ with $a_{j,k}=1$ for $j\leqslant{}k$ and zero otherwise, we have $c_0(A)=\mathbb{C}^\mathbb{N}$ carrying the topology of pointwise convergence. Remember that $\Re q_j<0$ holds for all $j$ and $\limsup_{j\rightarrow\infty}=0$ by our assumptions on $(q_j)_{j\in\mathbb{N}}$. The semigroup $\T$ is pseudo uniformly exponentially stable and thus also pseudo strongly uniformly exponentially stable: Given $k$ we select $0<\omega<|\max_{j=1,\dots,k}\Re q_j|$. For $x\in\mathbb{C}^\mathbb{N}$ we have
\begin{eqnarray*}
\|e^{\omega t}T(t)x\|_k &=& \max_{j=1,\dots,k}|x_j|\cdot{}e^{(\omega +\Re q_j)t}\\
&\leqslant& \|x\|_k \cdot \max_{j=1,\dots,k}e^{(\omega +\Re q_j)t}\\
&\leqslant& \|x\|_k \cdot e^{(\omega+\max_{j=1,\dots,k}\Re q_j)t}.
\end{eqnarray*}
Since $\sup_{x\in B}\|x\|_k<\infty$ holds for any bounded $B$ it follows that $\sup_{x\in B}\|e^{\omega{}t}T(t)x\|_k$ converges to zero for $t\rightarrow\infty$. On the other hand for $x=(1,1,\dots)$ and $\omega>0$ we can select $j_0$ such that $\omega+\Re q_{j_0}>0$ holds, for $k\geqslant j_0$ whence $\|e^{\omega{}t}T(t)x\|_k\geqslant e^{(\omega+\Re q_{j_0})t}$ does not converge to zero as $t\rightarrow\infty$. Thus, $T$ is not strongly exponentially stable and thus also not uniformly exponentially stable.
\end{example}

Next, we consider the space $s$ of rapidly decreasing sequences.

\begin{example}\label{MULT-5} For $A=(a_{j,k})_{j,k\in\mathbb{N}}$ with $a_{j,k}=j^k$ we have $c_0(A)=s$ endowed with the usual topology. We consider the multiplication semigroup $\T$ with a sequence $(q_j)_{j\in\mathbb{N}}$ such that $\Re(q_j)=-1/j$. Then, $\T$ is super polynomially stable: For given $k$ select $m=k+3$. Then, $a_{j,k}=j^k=j^{k+3}/j^{3}=a_{j,m}/j^3$. For $x\in s$ estimate
\begin{eqnarray*}
\|t^2T(t)x\|_k &=& \sup_{j\in\mathbb{N}}a_{j,k}|t^2e^{tq_j}x_j| = t^2 \sup_{j\in\mathbb{N}}a_{j,m}j^{-3}e^{-t/j}|x_j|\\
&\leqslant& t^2 \|x\|_m\sup_{j\in\mathbb{N}}j^{-3}e^{-t/j}\leqslant{}t^2\|x\|_m 27e^{-3}t^{-3}
\end{eqnarray*}
where the last estimate is true for all $t\geqslant t_0$, $t_0$ suitable large. Consequently, $\sup_{\|x\|_m\leqslant1}\|t^2T(t)x\|_k$ converges to zero for $t\rightarrow\infty$. This yields Theorem \ref{MAIN-2}(iii) with $\alpha=2$. On the other hand, $\T$ is not pseudo strongly exponentially stable: Let $k$ be given. Select $x$ such that $|x_j|>0$ holds for all $j$ and let $\omega>0$ be given. Select $j_0$ such that $\omega+\Re q_{j_0}>0$ holds. Then $\|e^{\omega{}t}T(t)x\|_k\geqslant{}a_{j_0,k}e^{(\omega+\Re q_{j_0})t}|x_{j_0}|$ does not tend to zero for $t\rightarrow\infty$.
\end{example}

The last conclusion of Example \ref{MULT-5} holds already under the assumption on $A$ that there exists $k$ such that $a_{j,k}>0$ holds for all $j$ and for all sequences $(q_j)_{j\in\mathbb{N}}$ with  $\liminf_{j\rightarrow\infty}\Re q_j\geqslant0$.

\begin{example}\label{MULT-6} We keep our assumption that $\Re q_j<0$ holds for all $j$ and assume $\lim_{j\rightarrow\infty} \Re q_j=0$. Then for the multiplication semigroup $\T$ on $c_0(A)$ the following are equivalent.
\begin{itemize}\smallskip
\item[(i)] $\displaystyle\forall\:k\in\mathbb{N}\;\exists\:m\in\mathbb{N}\colon\lim_{j\rightarrow\infty}\frac{a_{j,k}}{a_{j,m}}=0$.\smallskip
\item[(ii)] $\displaystyle\forall\:k\in\mathbb{N}\;\exists\:m\in\mathbb{N}\colon\lim_{j\rightarrow\infty} \sup_{\|x\|_m\leqslant1}\|T(t)x\|_k=0$.\vspace{3pt}
\end{itemize}
The condition in (i) is well-known but under different names and characterizes the Schwartz property for $c_0(A)$. In the sequel we use the notation of Bierstedt, Meise, Summers \cite[4.1]{BMS1982a} and say that $A$ satisfies (S) in this case. Condition (ii) implies uniform stability.
\smallskip
\\\textquotedblleft{}(i)$\Rightarrow$(ii)\textquotedblright{}: Let $k$ be given. Select $m$ as in (i). W.l.o.g.~we may assume that $m\geqslant k$ holds. Let $\epsilon>0$ be given. By (i) there exists $J$ such that $a_{j,k}/a_{j,m}<\frac{\epsilon}{2}$ holds for all $j>J$. Put $q_0=\max_{j=1,\dots,J}\Re q_j<0$. Now select $t_0\geqslant0$ such that $e^{tq_0}\leqslant \frac{\epsilon}{2}$. Let $t\geqslant t_0$ and $\|x\|_m\leqslant1$. Compute
\begin{eqnarray*}
\|T(t)x\|_k &=& \max_{j=1,\dots,J}a_{j,k}|x_j|e^{t\Re q_j} + \sup_{j>J}a_{j,k}|x_j|e^{t\Re q_j}\\
&\leqslant& e^{tq_0}+\sup_{j>J}a_{j,m}\frac{\epsilon}{2}|x_j|\leqslant\frac{\epsilon}{2} + \frac{\epsilon}{2}\|x\|_m\leqslant\epsilon
\end{eqnarray*}
which shows $\sup_{\|x\|_m\leqslant1}\|T(t)x\|_k\leqslant\epsilon$ for $t\geqslant t_0$ as needed. 
\smallskip
\\(ii)$\Rightarrow$(i)\textquotedblright{}: Assume that (i) does not hold. Then we have
$$
\exists\:k\;\forall\:m\;\exists\:C>0,\;(j_n)_{n\in\mathbb{N}}\subseteq\mathbb{N}\text{ with } j_n\stackrel{\scriptscriptstyle n\rightarrow\infty}{\longrightarrow}\infty\colon\:\frac{a_{j_n,k}}{a_{j_n,m}}\geqslant C.
$$
Select $k$ as in above. Let $m$ be arbitrary. Choose $C>0$ and $(j_n)_{n\in\mathbb{N}}$ as above and define $(x^{(n)})_{n\in\mathbb{N}}$ via
$$
x^{(n)}=(x_j^{(n)})_{j\in\mathbb{N}},\;\;x_j^{(n)}=\begin{cases}
\;1/a_{j,m} & \text{ if } j=j_n,\\
\;\hspace{10pt}0 & \text{ otherwise.}\\
\end{cases}
$$
By construction, $x^{(n)}\in\varphi\subseteq c_0(A)$ and $\|x^{(n)}\|_m=\sup_{j\in\mathbb{N}}a_{j,m}|x_j^{(n)}|=a_{j_n,m}\cdot1/a_{j_n,m}=1$ for every $n\in\mathbb{N}$. For $n\in\mathbb{N}$ put $t_n:=1/|\Re q_{j_n}|$. Since $\lim_{j\rightarrow\infty}\Re q_j=0$ we have $\lim_{n\rightarrow\infty}t_n=\infty$. For arbitrary $n$ we estimate
$$
\|T(t_n)x^{(n)}\|_k=\sup_{j\in\mathbb{N}}e^{t_n \Re q_j}\,a_{j,k}|x_j^{(n)}|=e^{1/|\Re q_{j_n}|\cdot\Re q_{j_n}}\,a_{j_n,k}\big|1/a_{j_n,m}\big|\geqslant\frac{C}{e}
$$
which yields $\sup_{\|x\|_m\leqslant1}\|T(t_n)x\|_k\geqslant C/e>0$ for all $n\in\mathbb{N}$. Contradiction.
\end{example}

The assumption of the next example is always satisfied if the K\"othe matrix $A$ satisfies (S).

\begin{example}\label{MULT-7} Assume that $A$ satisfies condition (M)
$$
\forall\:I\subseteq\mathbb{N},\,n\in\mathbb{N}\;\exists\:k\in\mathbb{N}\colon \inf_{j\in J}\frac{a_{j,n}}{a_{j,k}}=0,
$$
cf.~\cite[4.1]{BMS1982a}. Then the multiplication semigroup is uniformly stable: By Example \ref{MULT-3}, $\T$ is strongly stable. By the Dieudonn\'{e}-Gomes theorem, see \cite[27.9]{MeiseVogtEnglisch}, we have  that $c_0(A)=\lambda^{\infty}(A)$ is a Montel space. Whence the conclusion follows from Theorem \ref{MAIN-4}.
\end{example}

Our last example is a combination of Example \ref{MULT-6} and Example \ref{MULT-7}.

\begin{example}\label{MULT-8} Let $A$ be a K\"othe matrix which satisfies (M) but not (S), see \cite[27.21]{MeiseVogtEnglisch} for a concrete example, let $\Re q_j<0$ and $\lim_{j\rightarrow\infty}\Re q_j=0$. Then the multiplication semigroup $\T$ is uniformly stable by Example \ref{MULT-7}. On the other hand, condition (ii) of Example \ref{MULT-6} does not hold. The latter firstly implies that this condition is indeed strictly stronger than uniform stability, see Remark \ref{REM-6}. Secondly, $\T$ is not super polynomially stable, since Theorem \ref{MAIN-2}(iii) would imply Example \ref{MULT-6}(ii).
\end{example}

To conclude this example section let us review the scheme of properties from Theorem \ref{MAIN-0} and indicate at which positions our counterexamples from above are located.
\begin{center}
\includegraphics[width=12cm]{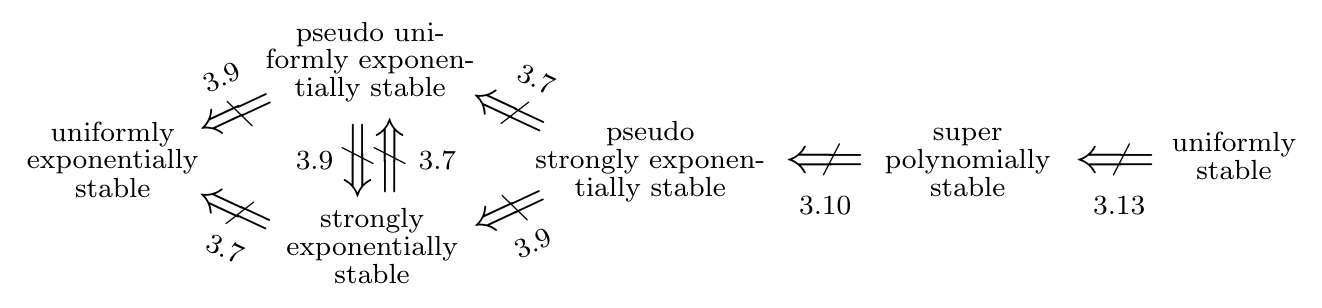}
\end{center}

\section{Proofs}\label{Proofs}

\subsection{Preparation}

In this section we consider families $\T\subseteq L(Y,X)$, where $X$ and $Y$ are locally convex spaces and $X$ is allowed to be non-Hausdorff. We refer to Floret, Wloka \cite[\S 23.1]{FloretWloka} for useful hints concerning the treatise of non-Hausdorff locally convex spaces and to Bourbaki \cite{Bourbaki} for non-Hausdorff versions of classical theorems. Possibly non-Hausdorff locally convex topologies $\tau_s$ resp.~$\tau_b$ are defined as at the beginning of Section \ref{NotationAndMainResults}, write $L_s(Y,X)$ resp.~$L_b(Y,X)$ for $L(Y,X)$ furnished with these topologies.
\smallskip
\\The first result below will yield the equivalences in Theorem's \ref{MAIN-1} and \ref{MAIN-2} for barrelled spaces by applying the latter to the sets $\mathcal{H}=\{t\mapsto e^{\omega t}\:;\:\omega>0\}$ resp.~$\mathcal{H}=\{t\mapsto t^{\alpha}\:;\:\alpha>1\}$. 

\begin{proposition}\label{PREP-1} Let $Y$ be barrelled, $X$ be a (not necessarily Hausdorff) locally convex space, let $\T\subset L(Y,X)$ be a locally equicontinuous subset and let $\mathcal{H}$ be a set of strictly positive functions $[0,\infty)\:\rightarrow\mathbb{R}$ which are bounded on compact intervals. We assume that for any $h\in\mathcal{H}$ there is $h'\in\mathcal{H}$ such that $h'/h$ vanishes at $\infty$. The following are equivalent.
\begin{itemize}
\item[(i)] $\exists\:h\in\mathcal{H}\;\forall\:q\in\Gamma_X\;\exists\:p\in\Gamma_Y\colon\lim_{t\rightarrow\infty}\sup_{p(y)\leqslant1}q(h(t)T(t)y)=0$.\vspace{3pt}
\item[(ii)] $\exists\:h\in\mathcal{H}\colon \lim_{t\rightarrow\infty}h(t)T(t)=0$ in $L_b(Y,X)$.\vspace{3pt}
\item[(iii)] $\exists\:h\in\mathcal{H}\colon \lim_{t\rightarrow\infty}h(t)T(t)=0$ in $L_s(Y,X)$.\vspace{3pt}
\item[(iv)] $\exists\:h\in\mathcal{H}\;\forall\:q\in\Gamma_X\;\exists\:p\in\Gamma_Y,\,M\geqslant1\;\forall\:t\geqslant0,\,y\in Y\colon q(h(t)T(t)y)\leqslant M p(y)$.\vspace{3pt}
\item[(v)] $\exists\:h\in\mathcal{H}\colon\{h(t)T(t)\:;\:t\geqslant0\}\subseteq L_b(Y,X)$ is bounded.
\end{itemize}
\end{proposition}
\begin{proof} \textquotedblleft{}(i)$\Rightarrow$(ii)\textquotedblright{}: Select $h\in\mathcal{H}$ as in (i). Let $q\in\Gamma_X$ and $B\in\mathcal{B}_Y$ be given. Select $p\in\Gamma_Y$ as in (i) with respect to $h$ and $q$. Since $B$ is bounded there is $C\geqslant0$ such that $B\subseteq\{y\in Y\:;\:p(y)\leqslant C\}$ holds. Thus, $\sup_{y\in B}q(h(t)T(t)y)\leqslant\sup_{p(y)\leqslant C}q(h(t)T(t)y)\leqslant C \sup_{p(y)\leqslant 1}q(h(t)T(t)y)$ implies (ii) taking into account the definition of $L_b(Y,X)$.
\smallskip
\\\textquotedblleft{}(ii)$\Rightarrow$(iii)\textquotedblright{}: Trivial.
\smallskip
\\\textquotedblleft{}(iii)$\Rightarrow$(iv)\textquotedblright{}: Select $h$ as in (iii). We claim that $\{h(t)T(t)\:;\:t\geqslant0\}$ is bounded in $L_s(Y,X)$. Let $y\in Y$ and $q\in\Gamma_X$ be given. By (iii) there exists $t_0\geqslant0$ such that $q(h(t)T(t)y)\leqslant1$ holds for all $t\geqslant t_0$. Since $\T$ is locally equicontinuous and $h$ is bounded on compact intervals there exists $C\geqslant0$ such that $q(h(t)T(t)y)\leqslant C$ for all $t\in[0,t_0]$. Thus, $q_y(h(t)T(t))=q(h(t)T(t)y)\leqslant C+1$ holds for all $t\geqslant0$ and the claim is established. Since $Y$ is barrelled, \cite[Chapter III, \S 3.6, Th\'{e}or\`{e}me 2]{Bourbaki} implies that $\{h(t)T(t)\:;\:t\geqslant0\}$ is equicontinuous in $L(Y,X)$.  Writing down the latter explicitly shows (iv).
\smallskip
\\\textquotedblleft{}(iv)$\Rightarrow$(v)\textquotedblright{}: Trivial.
\smallskip
\\\textquotedblleft{}(v)$\Rightarrow$(iv)\textquotedblright{}: Since $Y$ is barrelled, we may use again \cite[Chapter III, \S 3.6, Th\'{e}or\`{e}me 2]{Bourbaki} to obtain from (v) that there is $h\in\mathcal{H}$ such that $\{h(t)T(t)\:;\:t\geqslant0\}\subseteq L(Y,X)$ is equicontinuous, i.e., (iv) holds.
\smallskip
\\\textquotedblleft{}(iv)$\Rightarrow$(i)\textquotedblright{}: Select $h\in\mathcal{H}$ as in (iv). Select $h'\in\mathcal{H}$ such that $h'/h$ vanishes at $\infty$. Let $q\in\Gamma_X$ be given. Select $p\in\Gamma_Y$ and $M\geqslant0$ according to (iv) w.r.t.~$h$ and $q$. Let $t\geqslant0$ and $y\in Y$ be given. By (iv) we have $q(h(t)T(t)y)\leqslant M p(y)$ and obtain $q(h'(t)T(t)y)\leqslant Mp(y)h'(t)/h(t)$ for $t$ large enough. Finally we get $\sup_{p(y)\leqslant1}q(h'(t)T(t)y)\leqslant Mh'(t)/h(t)\rightarrow0$ for $t\rightarrow\infty$ which shows (i).
\end{proof}

For the next result of this section we need a version of the classical principle of condensation of singularities, see Lemma \ref{PREP-2} below. We give a full proof; however we note that the latter is just a patchwork of well-known arguments in the framework of the uniform boundedness principle, cf.~\cite[p.~63f]{Yosida}, Bourbaki \cite[Chapter III, \S 3, Exercise 15.b]{Bourbaki} and \cite[1.2.18]{BPC}.

\begin{lemma}\label{PREP-2}(Principle of condensation of singularities) Let $Y$ be a Baire space and $X$ be a (not necessarily Hausdorff) locally convex space. For $\nu\in\mathbb{N}$ and $\mu\in[0,\infty)$ let $T_{\nu,\mu}\in L(Y,X)$. The following are equivalent.
\begin{itemize}
\item[(i)] $\forall\:\nu\in\mathbb{N}\;\exists\:y\in Y\colon \{T_{\nu,\mu}y\:;\:\mu\geqslant0\}$ is unbounded in $X$.\vspace{3pt}
\item[(ii)] $\exists\:y\in Y\;\forall\:\nu\in\mathbb{N}\colon \{T_{\nu,\mu}y\:;\:\mu\geqslant0\}$ is unbounded in $X$.
\end{itemize}
\end{lemma}
\begin{proof} \textquotedblleft{}(i)$\Rightarrow$(ii)\textquotedblright{}: We show that $B=\big\{y\in Y\:;\:\forall\:\nu\in\mathbb{N}\colon\{T_{\nu,\mu}y\:;\:\mu\geqslant0\}\subseteq X \text{ is unbounded}\big\}$ is of second category in $Y$. We fix $\nu\geqslant0$ and claim that $B_\nu=\{y\in Y\:;\:\{T_{\nu,\mu}y\:;\:\mu\geqslant0\}\subseteq X \text{ is bounded}\}$ is of first category. Select $y_\nu$ such that $\{T_{\nu,\mu}y_\nu\:;\:\mu\geqslant0\}\subseteq X$ is unbounded. That is, there exists $q\in\Gamma_X$ such that $\sup_{\mu\geqslant0}q(T_{\nu,\mu}y_\nu)=\infty$. Assume now that $B_\nu$ is of second category and fix $\epsilon>0$. For $k\geqslant1$ define $A_k=\{y\in Y\:;\:\sup_{\mu\geqslant0}q(k^{-1}T_{\nu,\mu}y)\leqslant\epsilon\}$. Then, $B_\nu\subseteq \cup_{k\geqslant1}A_k$: Let $y\in B_\nu$, that is  $\{T_{\nu,\mu}y\:;\:\mu\geqslant0\}\subseteq X$ is bounded and thus there exists $K\geqslant0$ such that $\sup_{\mu\geqslant0}q(T_{\nu,\mu}y)\leqslant K$ holds. Therefore there is $k\geqslant1$ 
such that $\sup_{\mu\geqslant0}q(k^{-1}T_{\nu,\mu}y)\leqslant K/k\leqslant\epsilon$ which means that $y\in A_k$ holds. As $B_\nu$ is of second category the same is true for $\cup_{k\geqslant1}A_k$ whence there exists $k_0$ such that the interior of $A_{k_0}$ is non-empty. Thus, there exists $y_0\in Y$, $p\in\Gamma_Y$ and $\delta>0$ such that $p(y-y_0)\leqslant\delta$ for $y\in Y$ implies that $\sup_{\mu\geqslant0}q(k_0^{-1}T_{\nu,\mu}y)\leqslant\epsilon$ holds. Define $y=y_\nu\delta/p(y_\nu)+y_0$, i.e., $p(y-y_0)\leqslant\delta$ and thus $\sup_{\mu\geqslant0}q(k_0^{-1}T_{\nu,\mu}y_\nu)\leqslant[q(k_0^{-1}T_{\nu,\mu}y)+q(k_0^{-1}T_{\nu,\mu}y_0)] p(y_\nu)/\delta\leqslant 2\epsilon p(y_\nu)/\delta<\infty$, if $p(y_\nu)>0$. Otherwise, put $y=\delta{}y_\nu+y_0$ and obtain $\sup_{\mu\geqslant0}q(k_0^{-1}T_{\nu,\mu}y_\nu)<\infty$ analogously. In both cases the finiteness contradicts the selection of $y_\nu$ which establishes the claim. Since $\nu$ was arbitrary, it follows that $\cup_{\nu\in\mathbb{N}}B_\nu=\{y 
\in Y\:;\:\exists\:\nu\in\mathbb{N}\colon \{T_{\nu,\mu}y\:;\:\mu\geqslant0\}\subseteq X \text{ is bounded}\}\subseteq Y$ is also of first category. Since $Y$ is of second category in itself, $Y\backslash\cup_{\nu\in\mathbb{N}}B_\nu=\{y\in Y\:;\:\forall\:\nu\in\mathbb{N}\colon \{T_{\nu,\mu}y\:;\:\mu\geqslant0\}\subseteq X \text{ is unbounded}\}=B$ cannot also be of first category.
\smallskip
\\\textquotedblleft{}(ii)$\Rightarrow$(i)\textquotedblright{}: Trivial.
\end{proof}

The next result is the key to prove the equivalences in Theorem's \ref{MAIN-1} and \ref{MAIN-2} for Mackey complete and for Baire spaces. The model cases for $\mathcal{H}$ are again the sets $\{t\mapsto e^{\omega t}\:;\:\omega>0\}$ resp.~$\{t\mapsto t^{\alpha}\:;\:\alpha>1\}$.

\begin{proposition}\label{PREP-3} Let $Y$ be a Baire space and $X$ be a (not necessarily Hausdorff) locally convex space, $\T\subseteq L(Y,X)$ be a locally equicontinuous subset and $\mathcal{H}$ be a set of strictly positive functions $[0,\infty)\rightarrow\mathbb{R}$ which are bounded on compact intervals. We assume that there exists $(h_n)_{n\in\mathbb{N}}\subseteq\mathcal{H}$ such that for any $h\in\mathcal{H}$ there is $n\in\mathbb{N}$ such that $h_n/h$ vanishes at $\infty$. The following are equivalent.
\begin{itemize}
\item[(i)] $\forall\:y\in Y\;\exists\:h\in\mathcal{H}\colon\lim_{t\rightarrow\infty}h(t)T(t)y=0$.\vspace{3pt}
\item[(ii)] $\exists\:h\in\mathcal{H}\;\forall\:y\in Y\colon\lim_{t\rightarrow\infty}h(t)T(t)y=0$.
\end{itemize}
\end{proposition}
\begin{proof} \textquotedblleft{}(i)$\Rightarrow$(ii)\textquotedblright{}: Assume that (ii) is not true. Then we have
$$
\forall\:h\in\mathcal{H}\;\exists\:q\in\Gamma_Y,\,y\in Y,\,c>0\,(t_k)_{k\in\mathbb{N}}\subseteq[0,\infty),\,t_k\nearrow\infty\;\forall\:k\in\mathbb{N}\colon
$$
$$
q(h(t_k)T(t_k)y)\geqslant c.
$$
For given $n\in\mathbb{N}$ we consider $h_n\in\mathcal{H}$ and select $m$ such that $h_m/h_n$ vanishes at infinity. We choose $q$, $y$, $c$ and $(t_k)_{k\in\mathbb{N}}$ as in the condition above w.r.t.~$h_m$. We obtain
$$
\sup_{t\geqslant0}q(h_n(t)T(t)y)\geqslant\sup_{k\in\mathbb{N}}q(h_m(t_k)T(t_k)y)\frac{h_n(t_k)}{h_m(t_k)}\geqslant c\,\sup_{k\in\mathbb{N}}\frac{h_n(t_k)}{h_m(t_k)}=\infty.
$$
Consequently, we have shown that for each $n\in\mathbb{N}$ there exists $y\in Y$ such that $\{h_n(t)T(t)\:;\:t\geqslant0\}$ is unbounded in $X$. Now we apply Lemma \ref{PREP-2} with $T_{\nu,\mu}=h_{\nu}T(\mu)$ and get that there exists $y\in Y$ such that for all $n\in\mathbb{N}$ the set $\{h_n(t)T(t)\:;\:t\geqslant0\}$ is unbounded in $X$.
\smallskip
\\For $y$ as before we select $h$ as in (i). Then we select $n$ such that $h_n/h$ vanishes at infinity, i.e., there exists $t_1\geqslant0$ such that $h_n(t)\leqslant h(t)$ holds for all $t\geqslant t_1$. We remember that by the previous paragraph $\{h_n(t)T(t)y\:;\:t\geqslant0\}$ is unbounded. Let now $q\in\Gamma_Y$ be given. By (i) there exists $t_2\geqslant0$ such that $q(h(t)T(t)y)\leqslant1$ holds for all $t\geqslant t_2$. We put $t_0=\max(t_1,t_2)$. Then there exists $C\geqslant0$ such that $q(h_n(t)T(t)y)\leqslant C$ holds for all $t\in[0,t_0]$. Therefore, 
\begin{eqnarray*}
\sup_{t\geqslant0}q(h_n(t)T(t)y)&\leqslant& C+\sup_{t\in[t_0,\infty)}q(h_n(t)T(t)y)\\
&\leqslant& C + \sup_{t\in[t_0,\infty)}q(h(t)T(t)y)\leqslant C+1
\end{eqnarray*}
and $\{h_n(t)T(t)y\:;\:t\geqslant0\}$ is bounded as $q\in\Gamma_Y$ was arbitrary, which is a contradiction.
\smallskip
\\\textquotedblleft{}(ii)$\Rightarrow$(i)\textquotedblright{}: Trivial.
\end{proof}

\subsection{The Main Results}

Below we give consecutively the proofs of our five main theorems. In order to avoid forward references we will not stick to the ordering in Section \ref{NotationAndMainResults}.

\begin{proof}\textit{(of Theorem \ref{MAIN-1})} 1.~Let $X$ be barrelled. In order to get the equivalence of (i)--(iii) it suffices to apply Proposition \ref{PREP-1} with $\mathcal{H}=\{t\mapsto e^{\omega t}\:;\:\omega>0\}$ and to evaluate its first three conditions.
\medskip
\\\label{PRF-MAIN-1.2}2.~Let $X$ be Mackey complete.
\smallskip
\\\textquotedblleft{}(iv)$\Rightarrow$(v)\textquotedblright{}: Let $B\in\mathcal{B}_X$ be given. Since $X$ is Mackey-complete, $B$ is contained in a Banach disk $B_0$. Denote by $X_{B_0}$ the associated Banach space, i.e., $X_{B_0}=\lin B_0$ endowed with the Minkowski functional $\|\cdot\|_{B_0}$ of $B_0$ as norm. For $t\geqslant0$ we consider the restriction $T(t)|_{X_{B_0}}\colon X_{B_0}\rightarrow X$ and we denote it also by $T(t)$. Since the inclusion $X_{B_0}\rightarrow X$ is continuous \cite[8.4.D]{Jarchow}, it follows that $\T\subseteq{}L(X_{B_0},X)$ holds. Moreover, it follows that the latter set is locally equicontinuous. We apply Proposition \ref{PREP-3} with $Y=X$ and $\mathcal{H}=\{t\mapsto e^{\omega t}\:;\:\omega>0\}$. Thus, there exists $\omega>0$ such that $\lim_{t\rightarrow\infty}e^{\omega t}T(t)=0$ holds in $L_s(X_{B_0},X)$, i.e., Proposition \ref{PREP-1}(iii) is satisfied. We evaluate Proposition \ref{PREP-1}(ii) and obtain $\omega>0$ such that $\lim_{t\rightarrow\infty}e^{\omega t}T(t)$ 
holds in $L_b(X_{B_0},X)$. The latter implies that $q_B(e^{\omega t}T(t))=\sup_{x\in B}q(e^{\omega t}T(t)x)\leqslant\sup_{x\in B_0}q(e^{\omega t}T(t)x)=\sup_{\|x\|_{B_0}\leqslant1}q(e^{\omega t}T(t)x)$ tends to zero for $t\rightarrow\infty$.
\smallskip
\\\textquotedblleft{}(v)$\Rightarrow$(iv)\textquotedblright{}: Trivial.
\medskip
\\3.~Let $X$ be Baire.
\smallskip
\\\textquotedblleft{}(i)$\Rightarrow$(v)$\Rightarrow$(iv)\textquotedblright{}: Trivial.
\smallskip
\\\textquotedblleft{}(iv)$\Rightarrow$(iii)\textquotedblright{}. Apply Proposition \ref{PREP-3} with $Y=X$ and $\mathcal{H}$ as in the second part of this proof.
\smallskip
\\\textquotedblleft{}(iii)$\Rightarrow$(ii)$\Rightarrow$(i)\textquotedblright{}: Follows from 1.~since Baire spaces are barrelled.
\end{proof}

\begin{proof}\textit{(of Theorem \ref{MAIN-2})} In the sequel we use the abbreviations (iii$_{\forall}$) resp.~(iii$_{\exists}$) for the two conditions in (iii) and similar notation for the conditions in (iv) and (v).
\medskip
\\1.~\textquotedblleft{}(i)$\Rightarrow$(ii)$\Rightarrow$(iii$_{\exists}$)$\Rightarrow$(iv$_{\exists}$)\textquotedblright{} and \textquotedblleft{}(i)$\Rightarrow$(iii$_{\forall}$)$\Rightarrow$(iv$_{\forall}$)$\Rightarrow$(iv$_{\exists}$)\textquotedblright{}: Trivial.
\smallskip
\\\textquotedblleft{}(iv$_{\exists}$)$\Rightarrow$(v$_{\forall}$)\textquotedblright{}: Let $\alpha>0$ be as in (iv$_{\exists}$) and select $1<\alpha_0<\alpha$. Then for given $x\in X$ we have $\lim_{t\rightarrow\infty}(1+t)^{\alpha_0}T(t)x=0$ since $(1+t)^{\alpha_0}\leqslant{}t^{\alpha}$ holds if $t$ is suitable large. We apply Proposition \ref{PREP-1} with $Y=X$ and $\mathcal{H}=\{t\mapsto (1+t)^{\alpha}\:;\:\alpha>1\}$. By the above, Proposition \ref{PREP-1}(iii) holds and we evaluate Proposition \ref{PREP-1}(iv) to get
$$
\exists\:\alpha>1\;\forall\:q\in\Gamma_X\;\exists\:p\in\Gamma_X,\,N\geqslant1\;\forall\:t\geqslant0,\,x\in X\colon q((1+t)^{\alpha}T(t)x)\leqslant Np(x).
$$
Select $\alpha>1$ as above and let $\beta\geqslant1$ and $q\in\Gamma_X$ be given. Select $p$ and $N$ as above. Put $M=N^{\beta}/(\alpha\beta-1)$ and let $x\in X$ be given. Compute
\begin{eqnarray*}
\int_0^\infty q(T(t)x)^{\beta}dt
&\leqslant&
\int_0^\infty\Bigl(p(x)\frac{N}{(1+t)^{\alpha}}\Bigr)^{\beta}dt\\
&=& N^{\beta}\int_0^\infty (1+t)^{-\alpha\beta}dt\,p(x)^{\beta}\\
&=&N^{\beta}\lim_{r\rightarrow\infty}\frac{(1+t)^{1-\alpha\beta}}{1-\alpha\beta}\Bigr|_0^r\,p(x)^{\beta}\\
&=&
\frac{N^{\beta}}{1-\alpha\beta}\,(\lim_{r\rightarrow\infty} (1+r)^{1-\alpha\beta}-1)\,p(x)^{\beta}\\
&=& M p(x)^{\beta}<\infty,
\end{eqnarray*}
which shows (v$_{\forall}$) and even gives a more precise estimate of the integral, i.e.,
\begin{equation}\label{EQ0}
\forall\:\beta\geqslant1,\,q\in\Gamma_X\;\exists\:p\in\Gamma_X,\,M\geqslant1\;\forall\:x\in X\colon \int_0^\infty q(T(t)x)^{\beta}dt\leqslant Mp(x)^{\beta}.
\end{equation}
\textquotedblleft{}(v$_{\forall}$)$\Rightarrow$(v$_{\exists}$)\textquotedblright{}: Trivial.
\smallskip
\\\textquotedblleft{}(v$_{\exists}$)$\Rightarrow$(i)\textquotedblright{}: We select $\beta$ as in (v$_{\exists}$) and fix this number for the remainder of this part of the proof. For our first argument we need a vector valued space of Lebesgue integrable functions: Let $I=[0,\infty)$. Below all integrals are to be understood with respect to the Lebesgue measure on $I$. Define
$$
\mathcal{L}^\beta(I,X)=\big\{f\colon I\rightarrow X\:;\:f \text{ measurable and } q_{\beta}^{\star}(f)<\infty \text{ for all } q\in\Gamma_X\big\}
$$
where
$$
q_{\beta}^{\star}(f)=\Bigl(\int_Iq(f(t))^{\beta}dt\Bigr)^{1/\beta}.
$$
We consider the linear subspace $N_{\beta}=\cap_{q\in\Gamma_X}\ker q_{\beta}^{\star}$ of $\mathcal{L}^{\beta}(I,X)$ and define
$$
L^{\beta}(I,X)=\mathcal{L}^{\beta}(I,X)/N_{\beta}
$$
endowed with the Hausdorff topology given by the system $\Gamma_{L^{\beta}(I,X)}=\{q_{\beta}\:;\:q\in\Gamma_X\}$ where $q_{\beta}([f])=q_{\beta}^{\star}(f)$ for $[f]\in L^{\beta}(I,X)$. Now we define the operators $\mathcal{T}_n\colon X\rightarrow L^{\beta}(I,X)$ via $\mathcal{T}_nx=\chi_{[0,n]}(\cdot)T(\cdot)x$ for $n\in\mathbb{N}$, $x\in X$. Since $\T$ is strongly continuous, $t\mapsto \chi_{[0,n]}(t)T(t)x$ is measurable for $x\in X$ and $n\in\mathbb{N}$. Moreover, for $t\geqslant0$ we have $0\leqslant q(\chi_{[0,n]}(t)T(t)x)^{\beta}\leqslant q(T(t)x)^{\beta}$ and thus
$$
\int_Iq(\chi_{[0,n]}(t)T(t)x)^{\beta}dt\leqslant\int_0^{\infty}q(T(t)x)^{\beta}dt<\infty
$$
for every $q\in\Gamma_X$, $n\in\mathbb{N}$ and $x\in X$. Whence the operators $\mathcal{T}_n$ are well-defined and linear. In addition, the above estimate shows that $\sup_{n\in\mathbb{N}}q_{\beta}(\mathcal{T}_nx)<\infty$ holds for every $x\in X$, $q\in\Gamma_X$. By the uniform boundedness principle, see \cite[23.26]{MeiseVogtEnglisch}, we obtain that $\{\mathcal{T}_n\:;\:n\in\mathbb{N}\}$ is an equicontinuous subset of $L(X,L^{\beta}(I,X))$. That is, for a given $p\in\Gamma_X$ there exist $s\in\Gamma_X$ and $C\geqslant1$ such that
$$
\Bigl(\int_0^np(T(t)x)^{\beta}dt\Bigr)^{1/\beta}=\Bigl(\int_I p(\chi_{[0,n]}(t)T(t)x)^{\beta}dt\Bigr)^{1/\beta}=p_{\beta}(\mathcal{T}_nx) \leqslant C s(x)
$$
holds for all $n\in\mathbb{N}$ and all $x\in X$. Therefore, with $K=C^{\beta}$, we obtain
$$
\int_0^\infty p(T(t)x)^{\beta}dt \leqslant K s(x)^{\beta}
$$
for all $x\in X$. So far, we proved
\begin{equation}\label{EQ1}
\forall\:p\in\Gamma_X\;\exists\:s\in\Gamma_X,\,K\geqslant1\;\forall\:x\in X\colon\int_0^\infty p(T(t)x)^{\beta}dt\leqslant Ks(x)^{\beta}.
\end{equation}
Next, we show by induction that
\begin{equation}\label{EQ2}
\forall\:q\in\Gamma_X\;\exists\:r\in\Gamma_X,\,N\geqslant1\:\forall\:t\geqslant0,\,x\in X\colon t^nq(T(t)x)^{\beta}\leqslant Nr(x)^{\beta}
\end{equation}
holds for any $n\in\mathbb{N}$. We start with $n=0$. Let $q\in\Gamma_X$ be given. Select $p\in\Gamma_X$, $M\geqslant1$ and $\omega\in\mathbb{R}$ as in Definition \ref{DFN-SEMIGROUP}(iv), that is according to the exponential boundedness of $\T$; w.l.o.g.~we may assume that $\omega>0$ holds. Select $s\in\Gamma_X$ and $K\geqslant1$ as in \eqref{EQ1}. Select $r\in\Gamma_X$ and $C\geqslant1$ such that $\max(p,s)\leqslant Cr$ holds. Put $N=C^{\beta}\max(M^{\beta}e^{\beta\omega},\,M^{\beta}K\beta\omega/(1-e^{-\beta\omega}))$. Let $x\in X$ be given. For $t\in[0,1)$ we have $q(T(t)x)^{\beta}\leqslant M^{\beta}e^{\beta\omega t}p(x)^{\beta}\leqslant M^{\beta}e^{\beta\omega} p(x)^{\beta}\leqslant NC^{-\beta} p(x)^{\beta}\leqslant N r(x)^{\beta}$ by the estimate in Definition \ref{DFN-SEMIGROUP}(iv), the definition of $N$ and the selection of $r$. For $t\geqslant1$ we have
\begin{eqnarray*}
\frac{1-e^{-\beta\omega t}}{\beta\omega} q(T(t)x)^{\beta} & = & \int_0^t e^{-\beta\omega\tau}q(T(t)x)^{\beta}d\tau\\
&=& \int_0^t e^{-\beta\omega\tau}q(T(\tau)T(t-\tau)x)^{\beta}d\tau  \\
&\leqslant& M^{\beta}\int_0^t p(T(t-\tau)x)^{\beta}d\tau\\
&=& M^{\beta}\int_0^t p(T(\tau)x)^{\beta}d\tau\leqslant M^{\beta}Ks(x)^{\beta}
\end{eqnarray*}
by the estimate in Definition \ref{DFN-SEMIGROUP}(iv) and by \eqref{EQ1} and thus
$$
q(T(t)x)^{\beta}\leqslant M^{\beta}K\frac{\beta\omega}{1-e^{-\beta\omega t}}s(x)^{\beta}\leqslant NC^{-\beta}s(x)^{\beta}\leqslant Nr(x)^{\beta}
$$
which finishes the initial step of our induction. Let us now fix $n\in\mathbb{N}$ and assume that that \eqref{EQ2} holds. We claim that the latter condition then also is true for $n$ replaced with $n+1$ that is we claim
\begin{equation}\label{EQ3}
\forall\:q\in\Gamma_X\;\exists\:s\in\Gamma_X,\,M\geqslant1\:\forall\:t\geqslant0,\,x\in X\colon t^{n+1}q(T(t)x)^{\beta}\leqslant Ms(x)^{\beta}.
\end{equation}
Let $q\in\Gamma_X$ be given. Select $r\in\Gamma_X$ and $N\geqslant1$ as in \eqref{EQ2}. Select $s\in\Gamma_X$ and $K\geqslant1$ according to \eqref{EQ1} w.r.t.~$p=r$. Put $M=(n+1)NK$. Let $t\geqslant0$ and $x\in X$ be given. Compute
\begin{eqnarray*}
\frac{1}{n+1} t^{n+1}q(T(t)x)^\beta & = & \int_0^t (t-\tau)^{n}q(T(t)x)^{\beta}d\tau\\
&=& \int_0^t (t-\tau)^{n}q(T(t-\tau)T(\tau)x)^{\beta}d\tau\\
&\leqslant&  N\int_0^t r(T(\tau)x)^{\beta}d\tau\\
&\leqslant& N K s(x)^{\beta}.
\end{eqnarray*}
where we first use \eqref{EQ2} and then \eqref{EQ1} with $p=r$. Therefore,
$$
t^{n+1}q(T(t)x)^{\beta}\leqslant(n+1) N K s(x)^{\beta} = M s(x)^{\beta}
$$
holds, the induction step is finished and we have shown
$$
\forall\:n\in\mathbb{N},\,q\in\Gamma_X\;\exists\:p\in\Gamma_X,\,N\geqslant1\;\forall\:t\geqslant0,\,x\in X\colon  q(t^{n/\beta}T(t)x)\leqslant N^{1/\beta} p(x).
$$
Now we show (i), i.e., Definition \ref{DFN-STABILITY}(v). Let $\alpha>1$ be given and choose $\alpha_0>\alpha$. Select $n\in\mathbb{N}$ such that $n\geqslant\alpha_0\beta$. Let $q\in\Gamma_X$ and $B\in\mathcal{B}_X$ be given. Select $p\in\Gamma_X$ and $N\geqslant1$ as above w.r.t.~$n$ and $q$. Put $M=N^{1/\beta}$. For $x\in X$ and $t\geqslant1$ we get
$$
q(t^{\alpha_0}T(t)x)\leqslant q(t^{n/\beta}T(t)x)\leqslant N^{1/\beta} p(x)=M p(x)
$$
and thus $q(t^{\alpha}T(t)x)\leqslant{}t^{\alpha-\alpha_0}Mp(x)$. Since $K=\sup_{x\in B}p(x)<\infty$ we obtain $q_B(t^{\alpha}T(t))\leqslant M K t^{\alpha-\alpha_0}$ which implies $\lim_{t\rightarrow\infty}q_B(t^{\alpha}T(t))=0$ as desired.
\medskip
\\2.~Let $X$ be Mackey complete.
\smallskip
\\\textquotedblleft{}(vi)$\Rightarrow$(vii)\textquotedblright{}: Copy the proof of the second part of Theorem \ref{MAIN-1} verbatim, but replace $\mathcal{H}=\{t\mapsto e^{\omega{}t}\:;\:\omega>0\}$ in the latter proof with $\mathcal{H}=\{t\mapsto t^{\alpha}\:;\:\alpha>1\}$.
\smallskip
\\\textquotedblleft{}(vii)$\Rightarrow$(vi)\textquotedblright{}: Trivial.
\medskip
\\3.~Let $X$ be Baire. By 1., (i), (ii), (iii$_{\forall}$), (iii$_{\exists}$), (iv$_{\forall}$), (vi$_{\exists}$), (v$_{\forall}$) and (v$_{\exists}$) are equivalent.
\smallskip
\\\textquotedblleft{}(ii)$\Rightarrow$(vii)$\Rightarrow$(vi)\textquotedblright{}: Trivial.
\smallskip
\\\textquotedblleft{}(vi)$\Rightarrow$(iv$_{\exists}$)\textquotedblright{}: Follows from Proposition \ref{PREP-3} with $Y=X$ and $\mathcal{H}=\{t\mapsto t^{\alpha}\:;\:\alpha>1\}$.
\end{proof}

\begin{proof}\textit{(of Theorem \ref{MAIN-3})} \textquotedblleft{}(iv)$\Rightarrow$(iii)\textquotedblright{}: Trivial.
\smallskip
\\\textquotedblleft{}(iii)$\Rightarrow$(iv)\textquotedblright{}: Let $q\in\Gamma_X$ and $B\in\mathcal{B}_X$ be given. Select a Banach disk $B_0$ and denote by $X_{B_0}$ the associated Banach space. Denote by $X_q$ the seminormed space $(X,q)$. Consider for $t\geqslant0$ the maps $T(t)|_{X_{B_0}}\colon X_{B_0}\rightarrow X_q$ and denote them also with $T(t)$. Since the inclusion $X_{B_0}\rightarrow X$ and the identity $X\rightarrow X_q$ are both continuous it follows that $\T\subseteq L(X_{B_0},X_q)$ is locally equicontinuous. We apply Proposition \ref{PREP-3} with $\T$ as above and $\mathcal{H}=\{t\mapsto e^{\omega{}t}\:;\:\omega>0\}$. Thus, there is $\omega>0$ such that $\lim_{t\rightarrow\infty}e^{\omega{}t}T(t)=0$ holds in $L_s(X_{B_0},X_q)$, i.e., Proposition \ref{PREP-1}(iii) is satisfied. We evaluate Proposition \ref{PREP-1}(ii) and obtain $\omega>0$ such that $\lim_{t\rightarrow\infty}e^{\omega{}t}T(t)=0$ holds in $L_b(X_{B_0},X_q)$. The latter implies that $q_B(e^{\omega t}T(t))=\sup_{x\in B}q(
e^{\omega t}T(t)x)\leqslant\sup_{x\in B_0}q(e^{\omega t}T(t)x)=\sup_{\|x\|_{B_0}\leqslant1}q(e^{\omega t}T(t)x)$ tends to zero for $t\rightarrow\infty$.
\smallskip
\\\textquotedblleft{}(i)$\Rightarrow$(ii)\textquotedblright{}: Trivial.
\smallskip
\\\textquotedblleft{}(ii)$\Rightarrow$(i)\textquotedblright{}: Let $q\in\Gamma_X$ be given. Select $\omega>0$ as in (ii) and choose $0<\omega_0<\omega$. As above we consider the locally equicontinuous semigroup $\T\subseteq L(X_{B_0},X_q)$ where $B_0$ is a Banach disk containing $B$. By (ii), we know that $\lim_{t\rightarrow\infty}e^{\omega{}t}T(t)=0$ holds in $L_s(X_{B_0},X_q)$. An inspection of the proof of Proposition \ref{PREP-1} exhibits that for the implication \textquotedblleft{}(iii)$\Rightarrow$(ii)\textquotedblright{} of the latter exactly one passage from $h$ to $h'$ is necessary. Thus, in our situation it follows that $\lim_{t\rightarrow\infty}e^{\omega_0t}T(t)=0$ holds in $L_b(X_{B_0},X_q)$. With the same estimate as in the previous part it follows that $q_B(e^{\omega_0t}T(t))$ tends to zero for $t\rightarrow\infty$. 
\end{proof}

\begin{proof}\textit{(of Theorem \ref{MAIN-0})} In view of Definition \ref{DFN-STABILITY} and the scheme at the end of Section \ref{Examples} we only have to prove that in our setting of Theorem \ref{MAIN-0} a pseudo strongly exponentially stable semigroup is super polynomially stable: Let $X$ be Mackey complete and barrelled and $\T$ be exponentially bounded and pseudo strong\-ly exponentially stable. Thus, Theorem \ref{MAIN-3}(iv) holds. We claim that Theorem \ref{MAIN-2}(ii) is satisfied: Put $\alpha=2$ and let $B\in\mathcal{B}_X$ as well as $q\in\Gamma_X$ be given. Select $\omega>0$ as in Theorem \ref{MAIN-3}(iv) w.r.t.~$q$ and $B$. Compute $\lim_{t\rightarrow\infty}q_B(t^\alpha T(t))=\lim_{t\rightarrow\infty}t^2 e^{-\omega{}t} \lim_{t\rightarrow\infty}q_B(e^{\omega{}t}T(t))=0$.
\end{proof}

\begin{proof}\textit{(of Theorem \ref{MAIN-4})} \textquotedblleft{}(i)$\Rightarrow$(ii)\textquotedblright{}: Trivial.
\smallskip
\\\textquotedblleft{}(ii)$\Rightarrow$(i)\textquotedblright{}: Let $B\in\mathcal{B}_X$, $q\in\Gamma_X$ and $\epsilon>0$ be given. Since $X$ is by definition barrelled it follows that the bounded $C_0$-semigroup $\T$ is equicontinuous. Thus, we may select $p\in\Gamma_X$ and $C\geqslant0$ such that $q(T(t)x)\leqslant Cp(x)$ holds for any $x\in X$ and $t\geqslant0$.  Put $\epsilon_0=\epsilon/(C+1)$. The sets $B_{p}(x,\epsilon_0)=\{y\in X\:;\:p(x-y)<\epsilon_0\}$, $x\in \overline{B}$, form an open cover of $\overline{B}$. By the Montel property there exist $x_1,\dots,x_m\in\overline{B}$ with $\overline{B}\subseteq\cup_{n=1}^m B_{p}(x_n,\epsilon_0)$. For fixed $1\leqslant n\leqslant m$ there exists $t_{0,n}\geqslant0$ such that $q(T(t)x_n)<\epsilon_0$ holds for $t\geqslant t_{0,n}$. We put $t_0=\max_{n=1,\dots,m}t_{0,n}$. Let now $x\in B\subseteq \overline{B}$ be given. Then there exists $1\leqslant n\leqslant m$ such that $x\in B_{p}(x_n,\epsilon_0)$. For $t\geqslant t_0$ we have $q(T(t)x)\leqslant q(T(t)(x-x_
n))+q(T(t)x_n)\leqslant Cp(x-x_n)+p(T(t)x_n)\leqslant C\epsilon_0+\epsilon_0=\epsilon$. Since $x\in B$ was arbitrary we get $q_B(T(t))=\sup_{x\in B}q(T(t)x)\leqslant\epsilon$ for $t\geqslant t_0$. 
\end{proof}

\subsection{Remarks}

We start with remarks on our general setting.

\begin{remark}\label{REM-1}\textit{(Barrelledness vs.~continuity)} The proofs in the previous section show that for several single implications of our results certain assumptions can be dropped. Some of the results hold for families $\T\subseteq L(X)$ that do not enjoy the evolution property or are not strongly continuous. However, many of the latter implications require $\T$ at least to be locally bounded. If the space $X$ is quasibarrelled then the latter is equivalent to $\T$ being locally equicontinuous, see, e.g., \cite[11.2.7]{Jarchow}. If $\T$ has the evolution property then the following holds: Firstly, if $\T$ is locally equicontinuous, then strong continuity and strong continuity at zero are equivalent by \cite[Remark 1(iii)]{ABR}. Secondly, if $\T$ is strongly continuous and $X$ is barrelled, then $\T$ is locally equicontinuous by \cite[1.1]{Komura}. In view of these results it is on the one hand possible to transfer assumptions from the space to the semigroup and vice versa, at least for the 
proof 
of 
single implications of our results. On the other hand, the relations cited above show that mostly the additional requirement of one of the properties above already takes one back to the setting of a strongly continuous semigroup on a barrelled space.
\end{remark}

\begin{remark}\label{REM-2}\textit{(Mackey completeness)} In view of the question for a generator of a given semigroup $\T$ on a space $X$ a certain completeness assumption (e.g., sequential completeness) is natural and seems to be indispensable to establish its basic properties, see \cite[Section 1]{Komura}. We mention that every sequentially complete space is Mackey complete, see \cite[5.1.8]{BPC}. Concerning Remark \ref{REM-1} we remark that a Mackey complete space is barrelled if and only if it is quasibarrelled.
\end{remark}

\begin{remark}\label{REM-3}\textit{(Exponential boundedness)} Let us first mention that on a Banach space every $C_0$-semigroup is exponentially bounded in the sense of Definition \ref{DFN-SEMIGROUP}(iv). For locally convex spaces, Babalola \cite[2.7]{Babalola} considered a condition equivalent to Definition \ref{DFN-SEMIGROUP}(iv) in his definition of \textquotedblleft{}$L_A(X)$-semigroups of class $(C_0,1)$\textquotedblright{}. Vuvunikyan \cite[Definition on p.~203]{Vuvunikyan} allowed in Definition \ref{DFN-SEMIGROUP}(iv) that $M$ and $\omega$ may depend on $x$; he called semigroups of this type \emph{quasiexponential}. If $\omega$ can be selected independent of $q$, the semigroup is said to be \emph{exponentially equicontinuous}, see Albanese, Bonet, Ricker \cite[2.1(iii)]{ABR} or \emph{quasi-equicontinuous}, see Choe \cite[p.~294]{Choe}. Finally, we like to point out that an exponentially bounded semigroup is always locally equicontinuous, compare with Remark \ref{REM-1}, and that every 
equicontinuous 
semigroup is exponentially bounded. In view of Observation \ref{OBS-1} the latter can thus be considered as a necessary condition for reasonable uniform stability properties.
\end{remark}

We go on with remarks on our results.

\begin{remark}\label{REM-4}\textit{(Uniform exponential stability)} The proof of Theorem \ref{MAIN-1} has shown that for a barrelled space $X$ and a $C_0$-semigroup $\T$ also the property
\begin{equation}\label{EQ5}
\exists\:\omega>0\;\forall\:q\in\Gamma_X\;\exists\:p\in\Gamma_X,\,M\geqslant1\;\forall\:t\geqslant0,\,x\in X\colon q(T(t)x)\leqslant{}Me^{-\omega t}p(x)
\end{equation}
is equivalent to uniform exponential stability. The above is the definition of exponential resp.~quasi-equicontinuity (see Remark \ref{REM-3}) with strictly negative $\omega$.
\end{remark}

\begin{remark}\label{REM-5}\textit{(Super polynomial stability)} The proof of Theorem \ref{MAIN-2} has shown that for a barrelled space $X$ and an exponentially bounded $C_0$-semigroup $\T$ also the property
\begin{equation}
  \label{EQ6}
  \begin{gathered}
    \forall\text{ or, equivalently, }\exists\:\beta\geqslant1\;\forall\:q\in\Gamma_X\;\exists\:p\in\Gamma_X,\,M\geqslant1\;\forall\:x\in X\colon\\
    \int_0^\infty q(T(t)x)^{\beta}dt\leqslant Mp(x)^{\beta}
  \end{gathered}
\end{equation}
is equivalent to super polynomial stability, see \eqref{EQ0} and, with $\beta\geqslant1$ as selected at the beginning of \textquotedblleft{}(v$_{\exists}$)$\Rightarrow$(i)\textquotedblright{} in the proof of Theorem \ref{MAIN-2}, \eqref{EQ1}. Condition \eqref{EQ6} corresponds to the condition used by Datko \cite[p.~615]{Datko} in his original version of Theorem B in Section \ref{Intro}. In the proof mentioned above, we used an iteration argument starting with \eqref{EQ1}. In the classical situation of a Banach space, where $p=q=\|\cdot\|_X$, it is possible to keep track of the right hand side's constants during this iteration. A summation over all iterated estimates then ends up with the series representation of the exponential function and yields uniform exponential stability of $\T$. On locally convex spaces the following variant is possible:
\smallskip
\\Consider the following statements.
\begin{itemize}\smallskip
\item[(i)]$\displaystyle\exists\:\beta\geqslant1\;\forall\:q\in\Gamma_X\;\exists\:N\geqslant1\;\forall\:x\in X\colon \int_0^\infty q(T(t)x)^{\beta}dt\leqslant Nq(x)^{\beta}$.\smallskip
\item[(ii)]$\displaystyle\forall\:q\in\Gamma_X\;\exists\:p\in\Gamma_X,\,M\geqslant1,\,\omega>0\:\forall\:t\geqslant0,\,x\in X\colon q(T(t)x)\leqslant Me^{-\omega{}t}p(x)$.\smallskip
\end{itemize}
If (i) holds for some fundamental system $\Gamma_X$, then (ii) holds for any fundamental system $\Gamma_X$.
\smallskip
\\In order to see this, fix a fundamental system $\Gamma^0_X$ such that (i) holds. Observe that Definition \ref{DFN-SEMIGROUP}(iv) holds for $\Gamma_X^0$. Modify the ordering of the quantifies in the induction starting with \eqref{EQ1} to show
\begin{equation}
  \label{EQ4}
  \begin{gathered}
    \exists\:\beta\geqslant1\:\forall\:q_0\in\Gamma^0_X\;\exists\:p_0\in\Gamma^0_X,\,N\geqslant 1\;\forall\:n\in\mathbb{N},\,t\geqslant0,\,x\in X\colon\\
    t^n q_0(T(t)x)^{\beta}\leqslant n!N^{n+1} p_0(x)^{\beta}.
  \end{gathered}
\end{equation}
Let now $\Gamma_X$ be an arbitrary fundamental system and let $q\in\Gamma_X$ be given. Select $c_0\geqslant1$ and $q_0\in\Gamma^0_X$ such that $q\leqslant c_0q_0$ holds. Choose $p_0$ and $N$ as in \eqref{EQ4}. Put $\omega=(2N\beta)^{-1}$ and $M=(2N)^{1/\beta}\,c_0C_0$. Select $p\in\Gamma_X$ and $C_0\geqslant1$ such that $p_0\leqslant C_0p$ holds. Let $t\geqslant0$ and $x\in X$ be given. By the estimate in \eqref{EQ4} we have
\begin{eqnarray*}
\frac{1}{n!}\Bigl(\frac{t}{2N}\Bigr)^n q(T(t)x)^{\beta}
&\leqslant&\frac{1}{n!}\Bigl(\frac{t}{2N}\Bigr)^n c_0^{\beta}\,q_0(T(t)x)^{\beta}\\
&\leqslant&\frac{N}{2^n}\,c_0^{\beta}\,p_0(x)^{\beta}\\
&\leqslant{}&
\frac{N}{2^n}(c_0C_0)^{\beta}\,p(x)^{\beta}
\end{eqnarray*}
for any $n\in\mathbb{N}$. Summation over $n$ yields
\begin{eqnarray*}
e^{\frac{t}{2N}}q(T(t)x)^{\beta}
&=&\sum_{n=0}^{\infty}\frac{1}{n!}\Bigl(\frac{t}{2N}\Bigr)^n q(T(t)x)^{\beta}\\
&\leqslant&\sum_{n=0}^{\infty}\frac{N}{2^n}\,(c_0C_0)^{\beta}\,p(x)^{\beta}\\
&=&2N\,(c_0C_0)^{\beta}\,p(x)^{\beta},
\end{eqnarray*}
i.e.,
$$
q(T(t)x)
\leqslant
(2N e^{-\frac{t}{2N}})^{1/\beta}\,c_0C_0\,p(x)
=
(2N)^{1/\beta} e^{-\frac{t}{2N\beta}}\,c_0C_0\,p(x)
=
M e^{-\omega t}p(x)
$$
which shows (ii).
\smallskip
\\Let us add, that (ii) coincides with the definition of exponential boundedness (Definition \ref{DFN-SEMIGROUP}(iv)), but with strictly negative $\omega$. The latter implies that $\T$ is pseudo uniformly exponentially and -- even without Mackey completeness -- super polynomially stable. However, Example \ref{MULT-4} shows that (i) and (ii) are not sufficient for strong exponential stability -- even not on a Fr\'{e}chet space.
\end{remark}

\begin{remark}\label{REM-6}\textit{(Uniform stability)} The condition
\begin{equation}\label{EQ7}
\forall\:q\in\Gamma_X\;\exists\:p\in\Gamma_X\colon\lim_{t\rightarrow\infty}\sup_{p(x)\leqslant1}q(T(t)x)=0
\end{equation}
is sufficient for uniform stability of $\T$. If $X$ is a Banach space, the expression $\sup_{p(x)\leqslant1}q(T(t)x)$ is -- up to a constant -- the operator norm of $T(t)$. However, for infinite dimensional non-normed $X$ the set $\{x\in X\:;\:p(x)\leqslant1\}$ is never bounded by Kolmogoroff's theorem, see K\"othe \cite[p.~160]{KoetheI}. In the case of uniform exponential and super polynomial stability, cf.~\textquotedblleft{}(i)$\Leftrightarrow$(ii)\textquotedblright{} in Theorem \ref{MAIN-1} and \textquotedblleft{}(i)$\Leftrightarrow$(iii)\textquotedblright{} in Theorem \ref{MAIN-2}, the corresponding conditions are in fact equivalent. In the \textquotedblleft{}non-weighted\textquotedblright{} situation this equivalence, i.e., \textquotedblleft{}\eqref{EQ7}$\Leftrightarrow$ Definition \ref{DFN-STABILITY}(iv)\textquotedblright{}, is not even true for Fr\'{e}chet-Montel spaces, see Example \ref{MULT-8}.
\end{remark}

\section{Applications to Evolution Equations}\label{Applications}

\subsection{Transport equation on the Schwartz space}\label{TRANS-SCHWARTZ}

For a fixed function $q\colon\mathbb{R}\rightarrow\mathbb{C}$ let us consider the Cauchy problem
$$\CPa\;\;\;
\begin{cases}
\,\frac{\partial}{\partial{}t}u(t,x)=\frac{\partial}{\partial{}x}u(t,x)+q(x)u(t,x)\text{ for }t\geqslant0,\:x\in\mathbb{R},\\
\,\hspace{8.5pt}u(0,x)=u_0(x)\text{ for }x\in\mathbb{R}
\end{cases}
$$
where the initial value $u_0$ belongs to the space
$$
\SR=\big\{f\in C^{\infty}(\mathbb{R})\:;\:\forall\:k,\,n\in\mathbb{N}\colon\lim_{|x|\rightarrow\infty} x^kf^{(n)}(x)=0\big\}
$$
of rapidly decreasing functions which is endowed with the topology given by the system of seminorms $(\|\cdot\|_N)_{N\in\mathbb{N}}$, where
$\|f\|_N=\max_{k,n\leqslant N}\sup_{x\in\mathbb{R}}|x^k\frac{d^n}{dx^n}f(x)|$.
\smallskip
\\If $q\equiv0$, \cite[Examples 4.2 and 6.2]{Choe} provides that for any $u_0\in\SR$ there is a unique solution of $\CPa$ in $\SR$ and that all these solutions are given by the exponentially equicontinuous $C_0$-semigroup $\T$ of right shifts on $\SR$, i.e., $(T(t)f)(x)=f(x+t)$ for $t\geqslant0$, $x\in\mathbb{R}$ and $f\in\SR$. Since
$$
\|T(t)f\|_0=q_{0,0}(T(t)f)=\sup_{x\in\mathbb{R}}|f(x+t)|\geqslant|f(-t+t)|=|f(0)|
$$
holds for all $t\geqslant0$ it follows that $\lim_{t\rightarrow\infty}T(t)f=0$ cannot hold for all $f\in\SR$. Therefore the latter semigroup enjoys none of the stability properties studied in the previous sections.
\smallskip
\\If $q(x)\equiv q\in\mathbb{C}$, then the unique solutions of $\CPa$ are given by the exponentially equicontinuous $C_0$-semigroup $\T$ of scaled right shifts on $\SR$, i.e., $(T(t)f)(x)=\exp(qt)f(x+t)$ for $t\geqslant0$, $x\in\mathbb{R}$ and $f\in\SR$, cf.~\cite[Lemma 3.1]{Extra}. For $f\in\SR$, $N\in\mathbb{N}$ and $\omega\in\mathbb{R}$ we compute
\begin{eqnarray*}
\|e^{\omega{}t}T(t)f\|_N &=& \max_{k,n\leqslant{}N}\sup_{x\in\mathbb{R}}\big|x^k\frac{d^n}{dx^n}[e^{\omega{}t}e^{qt}f(x+t)]\big|\\
&=&e^{\omega{}t}|e^{qt}|\max_{k,n\leqslant{}N}\sup_{x\in\mathbb{R}}\big|[(x+t)-t]^k f^{(n)}(x+t)]\big|\\
&=&e^{(\omega+\Re{}q)t}\max_{k,n\leqslant{}N}\sup_{x\in\mathbb{R}}\big|\sum_{j=0}^{k}\mybinom{k}{j}(x+t)^j(-1)^{k-j}t^{k-j}f^{(n)}(x+t)\big|\\
&\leqslant& e^{(\omega+\Re{}q)t}\big(\max_{k\leqslant{}N}\sum_{j=0}^{k}\mybinom{k}{j} t^{k-j}\big)\big(\max_{j,n\leqslant{}N}\sup_{x\in\mathbb{R}}\big|(x+t)^jf^{(n)}(x+t)\big|\big)\\
&\leqslant&e^{(\omega+\Re{}q)t}(1+t)^{N}\|f\|_N
\end{eqnarray*}
which in view of Theorem \ref{MAIN-1}(iii) shows that $\T$ is uniformly exponentially stable if and only if $\Re{}q<0$ holds -- if $\Re q\geqslant0$ holds, a computation similar to the case $q\equiv0$ shows that $\T$ is not strongly stable.
\smallskip
\\Let now $q$ be non-constant. We assume that $q\in C^{\infty}(\mathbb{R})$ is real valued, $q(x)\leqslant0$ holds for all $x\in\mathbb{R}$ and $q$ and all its derivatives are bounded. If we rewrite $\CPa$ in the form $\ACP$ with $A=\frac{d}{dx}+M_q$, where $(M_qf)(x)=q(x)f(x)$ is the multiplication operator corresponding to $q$, it can be checked by straight forward computations that $(A,\SR)$ is the generator of the equicontinuous $C_0$-semigroup $\T$ defined by
$$
[T(t)f](x)=\exp\Bigl(\int_x^{x+t}q(\tau)d\tau\Bigr)f(x+t),
$$
for $f\in\SR$, $t\geqslant0$ and $x\in\mathbb{R}$, compare with \cite[Exercise III.1.17(5)]{EngelNagelOne}.

\begin{proposition}\label{TRANS-PROP} In the situation above consider
$$
\mu_q\colon[0,\infty)\rightarrow\mathbb{R},\;\; \mu_q(t)=\sup_{x\in\mathbb{R}}\int_x^{x+t}q(\tau)d\tau.
$$
Then the following holds.
\begin{itemize}
\item[(i)] If $\mu_q=\mathrm{\Omega}(t)$, i.e., $\liminf_{t\rightarrow\infty}\frac{|\mu_q(t)|}{t}>0$, then $\T$ is uniformly exponentially stable.\vspace{3pt}
\item[(ii)] If $\mu_q$ is bounded, then $\T$ is not strongly stable.
\end{itemize}
\end{proposition}
\begin{proof} (i) Let $f\in\SR$ and $N\in\mathbb{N}$ be given. For $j\in\mathbb{N}$ we put $A_j=\{\alpha\in\mathbb{N}^j\:;\:1\cdot\alpha_1+2\cdot\alpha_2+\cdots+j\cdot\alpha_j=j\}$ and compute with Fa\`a di Bruno's formula 
\begin{eqnarray*}
\|T(t)f\|_N & = & \max_{k,n\leqslant{}N}\sup_{x\in\mathbb{R}}\big|x^k\frac{d^n}{dx^n}\bigl(\exp\Bigl(\int_x^{x+t}q(\tau)d\tau\Bigr)f(x+t)\bigr)\big|\\
& = & \max_{k,n\leqslant{}N}\sup_{x\in\mathbb{R}}\,\Bigl|\,x^k\sum_{j=0}^{n}\mybinom{n}{j} \Bigl[\sum_{\alpha\in A_j}\mybinom{j}{\alpha} \exp\Bigl(\int_x^{x+t}q(\tau)d\tau\Bigr)\\
& & \hspace{15pt}\cdot\prod_{m=1}^{j}\Bigl(\frac{q^{(m-1)}(x+t)-q^{(m-1)}(x)}{m!}\Bigr)^{\alpha_m}\Bigr] f^{(n-j)}(x+t)\Bigr|\\
&\leqslant & \sup_{x\in\mathbb{R}}\,\Bigl|\,\exp\Bigl(\int_x^{x+t}q(\tau)d\tau\Bigr)\Bigr|\cdot{}C\cdot\max_{k,n\leqslant{}N}\max_{j\leqslant{}n}\sup_{x\in\mathbb{R}}\big|x^kf^{(n-j)}(x+t)\big|\\
&\leqslant & \sup_{x\in\mathbb{R}}\,\Bigl|\,\exp\Bigl(\int_x^{x+t}q(\tau)d\tau\Bigr)\Bigr|\cdot{}C\cdot\max_{k,n\leqslant{}N}\sup_{x\in\mathbb{R}}\big|[(x+t)-t]^kf^{(n)}(x+t)\big|\\
&\leqslant & e^{\mu_q(t)}\,C\,(1+t)^N\,\|f\|_N
\end{eqnarray*}
where
$$
C=\max_{t\geqslant0}\max_{n\leqslant{}N}\sum_{j=0}^{n}\mybinom{n}{j}\sum_{\alpha\in A_j}\mybinom{j}{\alpha}\prod_{m=1}^{j}\sup_{x\in\mathbb{R}}\,\Bigl|\frac{q^{(m-1)}(x+t)-q^{(m-1)}(x)}{m!}\Bigr|^{\alpha_m}<\infty
$$
holds by our assumptions on $q$ and since $\sup_{x\in\mathbb{R}}|q^{(m-1)}(x+t)-q^{(m-1)}(x)|\leqslant\sup_{x\in\mathbb{R}}|q^{(m-1)}(x+t)|+\sup_{x\in\mathbb{R}}|q^{(m-1)}(x)|\leqslant2\sup_{x\in\mathbb{R}}|q^{(m-1)}(x)|$ is true.
\smallskip
\\ By (i) there exists $t_0\geqslant0$ such that $\nu=\inf_{t\geqslant{}t_0}\frac{|\mu_q(t)|}{t}>0$ holds. We select $\omega=\nu/2>0$ and compute $\omega{}t+\mu_q(t)\leqslant\omega{}t-\nu{}t=-\omega{}t$ for $t\geqslant t_0$. Hence,
$$
\|e^{\omega{}t}T(t)f\|\leqslant e^{\omega{}t+\mu_q(t)}\,C\,(1+t)^N\,\|f\|_N\leqslant e^{-\omega{}t}\,C\,(1+t)^N\,\|f\|_N
$$
converges to zero for $t\rightarrow\infty$. The conclusion follows from Theorem \ref{MAIN-1}(iii).
\medskip
\\(ii) Let $\mu_q$ be bounded, i.e., there is $\nu\geqslant0$ such that $|\mu_q(t)|\leqslant{}\nu$ holds for all $t$. Then
\begin{eqnarray*}
\|T(t)f\|_0 &=& \sup_{x\in\mathbb{R}}\,\Bigl|\,\exp\Bigl(\int_x^{x+t}q(\tau)d\tau\Bigr)\,f(x+t)\Bigr|\\
&=&\sup_{x\in\mathbb{R}}e^{\mu_q(t)}\big|f(x+t)\big|\geqslant\sup_{x\in\mathbb{R}}e^{-\nu}\big|f(x+t)\big|=e^{-\nu}\|f\|_{\infty}.
\end{eqnarray*}
implies that $\lim_{t\rightarrow\infty}T(t)f=0$ can only hold if $f\equiv0$.
\end{proof}

The condition in Proposition \ref{TRANS-PROP}(i) is clearly satisfied, if $q$ is bounded away from zero. The first example below shows that $\T$ may fail all stability properties if $q$ is strictly negative but accumulating zero at infinity --- even if this accumulation happens arbitrarily slow.

\begin{example}\label{TRANS-EX-1} Assume that $q(x)<0$ holds for all  $x\in\mathbb{R}$ and that there is some $x_0$ such that $q|_{[x_0,\infty)}$ is monotonically increasing with $\lim_{x\rightarrow\infty}q(x)=0$. Then $\T$ is not strongly stable. Indeed, by the mean value theorem it follows that $\mu_q(t)=\sup_{x\in\mathbb{R}}\int_x^{x+t}q(\tau)d\tau=0$ for all $t\geqslant0$ and the conclusion follows from Proposition \ref{TRANS-PROP}(ii).
\end{example}

The second example in contrast illustrates that $\T$ may be uniformly exponentially stable even if $\limsup_{x\rightarrow\infty}q(x)=0$.

\begin{example}\label{TRANS-EX-2} Let $q(x)=\sin(x)-1$. Then $\T$ is uniformly exponentially stable: For $t>0$ and $x\in\mathbb{R}$ we have $\mu_q(t)=\sup_{x\in\mathbb{R}}\int_x^{x+t}q(\tau)d\tau= \sup_{x\in\mathbb{R}}(-\cos(x+t)+\cos(x)-t)\leqslant 2-t$. Whence, $\liminf_{t\rightarrow\infty}\frac{|\mu_q(t)|}{t}\geqslant\liminf_{t\rightarrow\infty}\frac{t-2}{t}=1$ and Proposition \ref{TRANS-PROP}(i) yields the conclusion.
\end{example}

\subsection{Heat equation on the Schwartz space and on Miyadera's space H}\label{HEAT}

In contrast to our discussion of the transport equation we consider the Cauchy problem
$$\CPb\;\;\;
\begin{cases}
\,\frac{\partial}{\partial{}t}u(t,x)=\frac{\partial^2}{\partial{}x^2}u(t,x)+q\,u(x,t)\text{ for }t\geqslant0,\:x\in\mathbb{R},\\
\,\hspace{8.5pt}u(0,x)=u_0(x)\text{ for }x\in\mathbb{R}
\end{cases}
$$
for the heat equation only with a constant perturbation, i.e., $q\in\mathbb{C}$ is fixed. For every $u_0\in\SR$ the unique solution of $\CPb$ is given by the semigroup $\T$ defined by $[T(t)f](x)=\exp(qt)[S(t)f](x)$ for $f\in\SR$, $t\geqslant0$ and $x\in\mathbb{R}$, where $\Se$ denotes the Gaussian semigroup defined by
$$
[S(t)f](x)=\frac{1}{\sqrt{4\pi{}t}}\int_{\mathbb{R}} e^{-\frac{(x-y)^2}{4t}}f(y)dy
$$
for $t>0$. The latter is an exponentially equicontinuous $C_0$-semigroup by \cite[Example 6.1]{Choe}. For the proof of the next result remember that $(S(t)f)^{\scriptscriptstyle\wedge}(\xi)=e^{-\xi^2t}\hat{f}(\xi)$ holds for all $t\geqslant0$, $\xi\in\mathbb{R}$ and $f\in\SR$ and that the Fourier transform $^{\scriptscriptstyle\wedge}\colon\SR\rightarrow\SR$ is an isomorphism.

\begin{proposition}\label{HEAT-SCHWARTZ-PROP} Let $\T$ be the scaled Gaussian semigroup on $\SR$.
\begin{itemize}
\item[(i)] If $\Re q<0$, then $\T$ is uniformly exponentially stable.\vspace{3pt}
\item[(ii)] If $\Re q\geqslant0$, then $\T$ is not strongly stable.
\end{itemize}
In particular, the Gaussian semigroup $\Se$ itself is not strongly stable, cf.~the situation on e.g., $L^p(\mathbb{R})$, $1<p<\infty$, studied by Arendt, Batty, B\'{e}nilan \cite[Proposition 3.1]{ABB1992}. 
\end{proposition}
\begin{proof} (i) We fix $N\in\mathbb{N}$, $f\in\SR$ and $t\geqslant1$. We select $M\in\mathbb{N}$ and $C>0$ such that $\|g\|_N\leqslant C \|\hat{g}\|_M$ holds for all $g\in\SR$. For arbitrary $t\geqslant1$ and with $g=T(t)f$ we obtain
\begin{eqnarray*}
\frac{1}{C} \|T(t)f\|_N & \leqslant & \max_{k,n\leqslant M}\sup_{\xi\in\mathbb{R}}|\xi^k\frac{d^n}{d\xi^n}(e^{qt-\xi^2t}f(\xi))|\\
&=& |e^{qt}| \max_{k,n\leqslant M} \sup_{\xi\in\mathbb{R}}\big|\xi^k\sum_{j=0}^{n}\mybinom{n}{j} \big[\frac{d^j}{d\xi^j}e^{-\xi^2t}\big] f^{(n-j)}(\xi)\big|\\
&=&e^{t\Re q}\max_{k,n\leqslant M} \sup_{\xi\in\mathbb{R}}\big|\sum_{j=0}^{n}\mybinom{n}{j} \big[\sum_{\alpha\in A_j}\mybinom{j}{\alpha}e^{-\xi^2t}\prod_{m=1}^{j}\Bigl(\frac{(-(\cdot)^2t)^{(m)}(\xi)}{m!}\Bigr)^{\alpha_m}\big] \xi^kf^{(n-j)}(\xi)\big|\\
&\leqslant& e^{t\Re q}\max_{k,n\leqslant M} \sup_{\xi\in\mathbb{R}}\big|\sum_{j=0}^{n}\mybinom{n}{j} \big[\sum_{\alpha\in A_j}\mybinom{j}{\alpha}e^{-\xi^2t}\Bigl(\frac{-2\xi t}{1}\Bigr)^{\alpha_1}\Bigl(\frac{-2t}{2}\Bigr)^{\alpha_2}\cdot1\big]\big|\\
& &\hspace{15pt}\cdot\max_{j\leqslant n}\sup_{x\in\mathbb{R}}|\xi^kf^{(n-j)}(\xi)|\\
&\leqslant& e^{t\Re q}\|f\|_M\max_{n\leqslant M}\sum_{j=0}^{n}\mybinom{n}{j} \big[\sum_{\alpha\in A_j}\mybinom{j}{\alpha}\sup_{\xi\in\mathbb{R}}e^{-\xi^2t}\,2^{\alpha_1}|\xi|^{\alpha_1}t^{\alpha_1+\alpha_2}\big]\\
&\leqslant& e^{t\Re q}\|f\|_M\,K\,t^{M}\,\sup_{\xi\in\mathbb{R}}e^{-\xi^2t}\,|\xi|^{M}\;\;=\;\;e^{t\Re q}\|f\|_M\,K\,t^{M/2}
\end{eqnarray*}
where $K=\max_{n\leqslant M}\sum_{j=0}^{n}\tbinom{n}{j} \sum_{\alpha\in A_j}\tbinom{j}{\alpha}\,2^{\alpha_1}$ and we again used Fa\`a di Bruno's formula, see Proposition \ref{TRANS-PROP}. Finally, we select $0<\omega<-\Re q$ and obtain that $\|e^{\omega{}t}T(t)f\|_N\leqslant e^{(\omega-\Re q)t}\|f\|_M\,KC\,t^{M/2}$ tends to zero for $t\rightarrow\infty$.
\medskip
\\(ii) We select $C>0$ and $N\in\mathbb{N}$ such that $\|\hat{g}\|_0\leqslant C^{-1}\|g\|_N$ holds for all $g\in\SR$. We select $f\in\SR$ such that $\hat{f}(0)=1$ and put in the above estimate $g=S(t)f$. Then we compute
\begin{eqnarray*}
\|S(t)f\|_N &\geqslant& C\|(S(t)f)^{\scriptscriptstyle\wedge}\|_0=C\|e^{-(\cdot)^2t}\hat{f}(\cdot)\|_0\\
&=& C\sup_{\xi\in\mathbb{R}}|e^{-\xi^2t}\hat{f}(\xi)|\geqslant C\hat{f}(0)=C>0
\end{eqnarray*}
and obtain $\|T(t)f\|_N=\|e^{qt}S(t)f\|_N\geqslant e^{t\Re q}\,C$, which tends to infinity for $t\rightarrow\infty$.
\end{proof}

\smallskip
Let us finally consider the heat equation $\CPb$ on a Fr\'{e}chet space introduced by Miyadera \cite[Section 6]{Miyadera}, see also Choe \cite[Example 6.4]{Choe}: Let $H$ be the space of all real-valued $C^{\infty}$--functions on $\mathbb{R}$ whose partial derivatives of all orders belong to $L^2(\mathbb{R})$. We endow $H$ with the topology given by all Sobolev norms, i.e., we consider the system of seminorms $(\|\cdot\|_n)_{n\in\mathbb{N}}$ with
$$
\|f\|_n=\sum_{j\leqslant{}n}\|f^{(j)}\|_{L^2(\mathbb{R})}=\sum_{j\leqslant{}n}\Bigl(\int_{\mathbb{R}}|f^{(j)}(x)|^2dx\Bigr)^{1/2}.
$$
Using Sobolev's embedding theorem, e.g., Adams \cite[Theorem 4.12]{Adams}, it follows that $H=\proj{n\in\mathbb{N}}H^n(\mathbb{R})$ holds where $H^n(\mathbb{R})$ denotes the $n$-th Sobolev space.
\smallskip
\\As on the Schwartz space, for every $u_0\in H$ the unique solution of $\CPb$ is given by the scaled Gaussian semigroup, which is also on $H$ an exponentially equicontinuous $C_0$-semigroup, cf.~\cite[p.~316]{Choe}.

\begin{proposition}\label{HEAT-MIYADERA-PROP} Let $\T$ be the scaled Gaussian semigroup on $H$. 
\begin{itemize}
\item[(i)] If $\Re q<0$, then $\T$ is uniformly exponentially stable.\vspace{3pt}
\item[(ii)] If $\Re q\geqslant0$, then $\T$ is not uniformly stable.\vspace{3pt}
\item[(iii)] If $\Re q\leqslant0$, then $\T$ is strongly stable.
\end{itemize}
\end{proposition}
\begin{proof}For $n\in\mathbb{N}$ we define the space
$$
\hat{H}^{n}(\mathbb{R})=\big\{f\in L^2(\mathbb{R})\:;\:|f|_{n}^2=\int_{\mathbb{R}}(1+|x|^2)^{n}|f(x)|^2dx<\infty\big\}
$$
endowed with the norm $|\cdot|_{(n)}$. Then, the Fourier transform $\mathcal{F}\colon H^n(\mathbb{R})\rightarrow \hat{H}^{n}(\mathbb{R})$ is an isomorphism, we have $\hat{H}^{n+1}\subseteq \hat{H}^{n}$ with continuous inclusion and the diagram
\begin{diagram}[height=2.2em,width=2.3em]
\cdots &\rTo &                     &                    & \hat{H}^{n+1}(\mathbb{R})  &                         & \rTo                   &                      & \hat{H}^{n}(\mathbb{R}) &                         & \rTo       &      & \cdots \\
       &     &                     &\ruTo^{\mathcal{F}} &                            &\rdTo^{\mathcal{F}^{-1}} &                        & \ruTo^{\mathcal{F}}  &                         &\rdTo^{\mathcal{F}^{-1}} &            &      &        \\
\cdots &\rTo & H^{n+1}(\mathbb{R}) &                    & \rTo                       &                         & H^{n}(\mathbb{R})      &                      & \rTo                    &                         & H^{n-1}(\mathbb{R}) & \rTo & \cdots \\
\end{diagram}
with inclusion maps at any unlabeled arrow, commutes. We define $\hat{H}=\proj{n}\hat{H}^n$ and get that $\mathcal{F}\colon H\rightarrow \hat{H}$ is an isomorphism and it is enough to consider the semigroup $(\hat{T}(t))_{t\geqslant0}$, $[\hat{T}(t)f](x)=e^{qt-x^2t}f(x)$ for $f\in\hat{H}$, $t\geqslant0$ and $x\in\mathbb{R}$.
\smallskip
\\(i) We have
\begin{eqnarray*}
|\hat{T}(t)f|_n^2 &=& \int_{\mathbb{R}}(1+|x|^2)^n |e^{qt-x^2t}f(x)|^2dx\\
&\leqslant& e^{2(\Re q)t}\int_{\mathbb{R}}(1+|x|^2)^n|f(x)|^2dx\leqslant e^{2(\Re q)t}|f|_n^2,
\end{eqnarray*}
i.e., $|\hat{T}(t)f|_n\leqslant e^{(\Re q)t}|f|_n$ for every $n\in\mathbb{N}$ and thus uniform exponential stability follows.
\medskip
\\(ii) We define $B=\{f_k\:;\:k\in\mathbb{N}\}\subseteq L^2(\mathbb{R})$ via $f_k(x)=k^{1/2}$ for $x\in[0,1/k]$ and zero otherwise. Then, $|f_k|_n^2=\int_{\mathbb{R}}(1+|x|^2)^n|f_k(x)|^2dx\leqslant{}\int_0^{1/k}\,2^n\,|k^{1/2}|^2dx=2^n$ holds for every $n$ and $B\subseteq\hat{H}$ is bounded. For arbitrary $k$ we have $|\hat{T}(\frac{k}{2})f_k|^2_0=\int_{\mathbb{R}}|e^{-x^2\frac{k}{2}}f_k(x)|^2dx =\int_{\mathbb{R}}e^{-x^2 k}|k^{1/2}|^2dx\geqslant{}k\int_0^{1/k}e^{-xk}dx=\frac{e-1}{e}$
which provides that $\sup_{f\in B}|\hat{T}(t)f|_0$ cannot converge to zero for $t\rightarrow\infty$.
\medskip
\\(iii) It remains to check strong stability. We fix $f\in\hat{H}$ and $n\in\mathbb{N}$. Then we have
\begin{eqnarray*}
|\hat{T}(t)f|_n^2 &=& \int_{\mathbb{R}}(1+|x|^2)^n |e^{qt-x^2t}f(x)|^2dx\\
&\leqslant&\int_{\mathbb{R}}e^{-2x^2t}|f(x)|^2dx + \sum_{k=1}^{n}\mybinom{n}{k}\sup_{x\in\mathbb{R}}|x|^{k} e^{-2x^2t}\int_{\mathbb{R}}|f(x)|^2dx\\
&= & \|e^{-(\cdot)^2t}f\|_{L^2(\mathbb{R})}^2 + \sum_{k=1}^{n}\mybinom{n}{k}\,t^{-k/2}\,\Bigl(\frac{k}{2e}\Bigr)^{k/2}\,\|f\|_{L^2(\mathbb{R})}^2
\end{eqnarray*}
which converges to zero for $t\rightarrow\infty$; for the second summand this is clear and for the first see \cite[Proposition 3.1]{ABB1992}.
\end{proof}

\begin{remark}\label{MULTIPLICATION-REM} In this section we in fact studied a multiplication semigroup on $\SR$ resp.~$\hat{H}$. Its exponent was determined by the differential equation $\CPb$. However, a general study of the multiplication semigroup $\T$ on $\SR$ defined by $[T(t)f](x)=\exp(q(x)t)f(x)$ for $f\in\SR$, $x\in\mathbb{R}$ and $t\geqslant0$ is possible: If for instance $q$ is a $C^{\infty}$--function which attains only non-positive real values and if $q$ and all its derivatives are bounded, then straight forward computations show that $\T$ is a $C_0$-semigroup with generator $(M_q,\SR)$. Similar to Example \ref{MULT-5} it is then possible to construct $q$, e.g., $q(x)=-1/|x|$ for $|x|\geqslant1$ and suitable in between, such that $\T$ is polynomially stable but not uniformly exponentially stable. 
\end{remark}

\section*{Acknowledgements}
The authors like to thank the referee for his careful work. In addition they like to thank L.~Frerick, T.~Kalmes and J.~Wengenroth for pointing out a mistake in an earlier version of Theorem \ref{MAIN-1} and for explaining Example \ref{MULT-2} to them.


\end{document}